\documentclass[11pt,oneside,english]{amsart}
\usepackage[T1]{fontenc}
\usepackage[latin9]{inputenc}
\usepackage{geometry}
\usepackage{textcomp}
\usepackage{amstext}
\usepackage{amsthm}
\usepackage{amssymb}

\makeatletter
\numberwithin{equation}{section}
\numberwithin{figure}{section}
\theoremstyle{plain}
\newtheorem{thm}{\protect\theoremname}[section]
\theoremstyle{definition}
\newtheorem{example}[thm]{\protect\examplename}
\theoremstyle{plain}
\newtheorem{conjecture}[thm]{\protect\conjecturename}
\theoremstyle{plain}
\newtheorem{cor}[thm]{\protect\corollaryname}
\theoremstyle{plain}
\newtheorem{lem}[thm]{\protect\lemmaname}
\theoremstyle{plain}
\newtheorem{prop}[thm]{\protect\propositionname}
\theoremstyle{remark}
\newtheorem{rem}[thm]{\protect\remarkname}

\def\makebbb#1{
    \expandafter\gdef\csname#1\endcsname{
        \ensuremath{\Bbb{#1}}}
}\makebbb{R}\makebbb{N}\makebbb{Z}\makebbb{C}\makebbb{H}\makebbb{E}\makebbb{H}\makebbb{P}\makebbb{B}\makebbb{Q}\makebbb{E}

\usepackage{babel}

\usepackage{babel}

\makeatother

\usepackage{babel}

\makeatother

\usepackage{babel}
\providecommand{\conjecturename}{Conjecture}
\providecommand{\corollaryname}{Corollary}
\providecommand{\examplename}{Example}
\providecommand{\lemmaname}{Lemma}
\providecommand{\propositionname}{Proposition}
\providecommand{\remarkname}{Remark}
\providecommand{\theoremname}{Theorem}

\begin{document}
\title{Conical Calabi-Yau metrics on toric affine varieties and convex cones}
\author{Robert J. Berman,}
\begin{abstract}
It is shown that any affine toric variety $Y,$ which is $\Q-$Gorenstein,
admits a conical Ricci flat Kähler metric, which is smooth on the
regular locus of $Y.$ The corresponding Reeb vector is the unique
minimizer of the volume functional on the Reeb cone of $Y.$ The case
when the vertex point of $Y$ is an isolated singularity was previously
shown by Futaki\textendash Ono\textendash Wang. The proof is based
on an existence result for the inhomogeneous Monge-{}-Ampère equation
in $\R^{m}$ with exponential right hand side and with prescribed
target given by a proper convex cone, combined with transversal a
priori estimates on $Y.$
\end{abstract}

\maketitle

\section{Introduction}

The problem of finding canonical metrics on a (polarized) complex
projective algebraic manifold $X$ has a long and rich history leading
up to the Yau\textendash Tian\textendash Donaldson conjecture, which
reduces the existence problem to verifying a purely algebro-geometric
property called K-stability. In the case of Fano manifolds the conjecture
was settled rather recently (see the survey \cite{do2}). The projective
framework has also been generalized to the the setting of singular
affine varieties $Y$ \cite{c.s0,C-S,d-s,d-sII,l-l-x}. Apart from
the complex-geometric motivations - coming, in particular, from the
study of tangent cones at log terminal singularities and connections
to the Minimal Model Program (as discussed in the survey \cite{l-l-x})
- an important motivation for this generalization comes from the AdS/CFT
correspondence in theoretical physics, relating geometry to conformal
field theory, or more precisely superconformal gauge theories \cite{m-s-y,m-s-y0,g-m-s-y}.

\subsection{\label{subsec:Background}Background}

We will be mainly concerned with the case of\emph{ toric} affine Calabi\textendash Yau
varieties, but we start by recalling some general background. Let
$Y$ be a complex affine normal variety of dimension $m,$ which is
$\Q$-Gorenstein (i.e. a reflexive tensor power of the canonical sheaf
of $Y$ is a well-defined line bundle on $Y$). Assume that $Y$ is
endowed with a\emph{ good action} by a compact torus $T.$ This means
that the complexification $T_{\C}$ acts holomorphically and effectively
on $Y$ and there is a unique fixed point $y_{0}$ such that that
the closure of any orbit of $T_{\C}$ in $Y$ contains $y_{0}$ \cite{l-l-x}.
The point $y_{0}$ in $Y$ is called the\emph{ vertex point} of $(Y,T).$
An element $\xi$ in the Lie algebra of $T$ is said to be a \emph{Reeb
vector} if the weights $\lambda_{i}$ of the action by $\xi$ on the
coordinate ring $\mathcal{R}(Y)$ of $Y$ (or the ring of holomorphic
functions on $Y$) are non-negative and vanish only on constant functions.
An affine variety $(Y,\xi)$ decorated with a Reeb vector $\xi$ is
called a \emph{polarized affine variety} in \cite{c.s0}. More generally,
a vector field $\xi$ on $Y$ is said to be a Reeb vector field if
it is the infinitesimal generator of some Reeb vector. We will denote
by $T_{\xi}$ the corresponding minimal torus (in other words, the
orbits of $T_{\xi}$ in $Y$ coincide with the closure of the orbits
of the flow of $\xi$). The \emph{rank }of $\xi$ is defined as the
rank of the corresponding minimal torus $T_{\xi}$ and $\xi$ is said
to be \emph{quasi-regular} if it has rank one and otherwise $\xi$
is said to be \emph{irregular. }Equivalently, $\xi$ is quasi-regular
iff the leaf space of the holomorphic foliation of $Y,$ defined by
the complexification of $\xi,$ is Hausdorff.\emph{ }

In the case when $y_{0}$ is an isolated singularity in $Y$ it was
shown in \cite{C-S} that there exists a \emph{Calabi\textendash Yau
metric} on $Y_{reg}(=Y-\{y_{0}\})$, i.e. a Ricci flat Kähler metric
$\omega,$ which is\emph{ conical with respect to $\xi,$} i.e., 
\begin{equation}
\text{\ensuremath{\text{Ric \ensuremath{\omega}=}}0},\,\,\,\mathcal{L}_{-J\xi}\omega=2\omega\label{eq:def of conical cy intro}
\end{equation}
 iff $(Y,T,\xi)$ is\emph{ K-stable} for any torus $T$ containing
$\xi$ in its Lie algebra. Here $J$ denotes the complex structure
on $Y_{reg}$ and $\mathcal{L}_{-J\xi}$ denotes the Lie derivative
along the vector field on $Y_{reg}$ defined by $-J\xi,$ which generates
a $\R^{*}$-action on $Y$ with repulsive fixed point $y_{0}.$ The
K-stability of $(Y,T,\xi)$ is the affine analog of the notion of
K-stability of projective Fano manifolds, as appearing in the ordinary
projective Yau\textendash Tian\textendash Donaldson conjecture. It
amounts to a positivity condition for all special test configurations
(i.e. equivariant deformations) of $(Y,T,\xi).$ We also recall that
the conical Calabi\textendash Yau condition \ref{eq:def of conical cy intro}
equivalently means that $\omega$ is a metric cone over the link $(Y-\{0\})/\R^{*}$
of the affine singularity $Y,$ endowed with a smooth Sasaki\textendash Einstein
metric \cite{C-S}. Remarkably, already in the case when the link
is the topological five-sphere, there is an infinite family of Sasaki\textendash Einstein
metrics \cite{C-S,b-g-k}. 

The conical condition on the Kähler metric $\omega$ equivalently
means that $\omega$ may be expressed as 
\[
\omega=dd^{c}(r^{2}),\,\,\,dd^{c}:=\frac{i}{2\pi}\partial\bar{\partial}
\]
 for a function $r$ which is radial with respect to the corresponding
$\R^{*}$-action on $Y,$ i.e. $r$ is positive and 1-homogeneous
with respect to $\R^{*}.$ As a consequence, when $y_{0}$ is an isolated
singularity the function $r^{2}$ is automatically bounded (and even
continuous) in a neighborhood of $y_{0}.$ The condition that $\omega$
be Ricci flat equivalently means that $r^{2}$ solves the Calabi\textendash Yau
equation 
\begin{equation}
\left(dd^{c}(r^{2})\right)^{m}=i^{m^{2}}\Omega\wedge\bar{\Omega},\label{eq:CY intro}
\end{equation}
 where $\Omega$ is a nowhere vanishing $T$-equivariant (multi-valued)
holomorphic top form on $Y_{reg}.$ Any $(Y,T)$ admits such a form
$\Omega$ and it is uniquely determined, up to multiplication by a
complex constant. As a consequence, if there exists a Ricci flat Kähler
metric on $Y$ which is conical with respect to the Reeb vector $\xi,$
then $\xi$ has to satisfy the normalization condition 
\[
\mathcal{L}_{\xi}\Omega=im\Omega.
\]
As first shown in \cite{m-s-y0,m-s-y}, motivated by the AdS/CFT correspondence,
another more subtle condition on $\xi$ is that $\xi$ minimizes the
volume functional $V$ on the space of all normalized Reeb vectors: 

\begin{equation}
V(\xi):=\lim_{t\rightarrow0^{+}}t^{m}\sum_{i}e^{-t\lambda_{i}},\label{eq:V xi general intro}
\end{equation}
 where the sum runs over the infinite number strictly positive weights
$\lambda_{i}$ of $\xi,$ including multiplicity. As shown in \cite{c.s0}
this volume condition is also directly implied by the K-stability
of $(Y,T,\xi)$ (the corresponding test configuration is a product
test configuration). 
\begin{example}
Let $X$ be an $(m-1)$-dimensional Fano manifold, i.e. a compact
complex manifold whose anticanonical line bundle $-K_{X}$ is positive.
Then the $m$-dimensional variety $Y$ obtained by blowing down the
zero-section in the total space of $K_{X}\rightarrow X$ is an affine
Gorenstein variety endowed with a Reeb vector $\xi$, which is regular
(generating the standard $\C^{*}-$action along the fibers of $K_{X})$.
Moreover, a conical Calabi\textendash Yau metric $\omega_{Y}$ on
$(Y,\C^{*})$ corresponds to a Kähler metric $\omega_{X}$ on $X$
with constant positive Ricci curvature, defined as the ``horizontal''
part of $\omega_{Y}$ with respect to the fibration $Y\rightarrow X.$
However, even when $X$ does not admit any Kähler metric $\omega_{X}$
with constant positive Ricci curvature, it frequently happens that
there is another Reeb vector $\xi$ which does admit a conical Calabi\textendash Yau
metric. This is always the case when $X$ (and hence $Y)$ is toric,
in which case $\xi$ is obtained by minimizing the volume functional
on the Reeb cone of $Y$ \cite{f-o-w,c-f-o}. For example, when $X$
is $\P^{2}$ blown up in one or two points the corresponding ``minimal''
Reeb vector $\xi$ on $Y$ is irregular. Irregular conical Calabi\textendash Yau
metrics were first constructed in the physics literature as special
cases of a family $Y^{p,q}$ of explicit non-homogeneous metrics \cite{g-m-s-w}.
In fact, as shown in\cite{m-s}, these are all toric and $Y^{2,1}$
corresponds to $\P^{2}$ blown up in one point. 
\end{example}

In the general case, where the vertex point $y_{0}$ of $Y$ is not
assumed to be an isolated singularity, $\omega$ is said to be a\emph{
conical Calabi\textendash Yau metric for $(Y,T,\xi)$} if $\omega$
is a smooth conical Calabi\textendash Yau metric on $Y_{reg}$ and
moreover the corresponding Kähler potential $r^{2}$ is bounded in
a neighborhood of $y_{0}$ (see \cite{d-sII,l-w-x,l-l-x}). Such metrics
arise, in particular, on the tangent cones of Kähler\textendash Einstein
metrics with log terminal (klt) singularities \cite{d-sII}. The conjectural
extension of the result in \cite{C-S} to a non-isolated singularity
$y_{0}$ can thus be formulated as the following generalized affine
Yau\textendash Tian\textendash Donaldson conjecture:
\begin{conjecture}
(YTD) Let $Y$ be a normal affine variety which is $\Q$-Gorenstein,
endowed with a good action by a compact torus $T.$ Let $\xi$ be
a Reeb vector in the Lie algebra of $T.$ Then there exists a Calabi\textendash Yau
metric which is conical with respect to $\xi$ iff $(X,T,\xi)$ is
K-stable.
\end{conjecture}

By \cite[Prop 4.8]{d-sII} such a Calabi\textendash Yau metric is
uniquely determined modulo the action of the group $\text{Aut \ensuremath{(X,T)_{0}}}$
of automorphisms of $Y,$ commuting with $T$ and homotopic to the
identity element. A version of the Yau-Tian-Donaldson for singular\emph{
projective} Fano varieties, involving the stronger uniform version
of K-stability, has recently been established in \cite{l-t-w,li2}. 

Note that we are adopting the terminilogy for K-stability used in
\cite{C-S}, which corresponds to the notion of K-\emph{poly}stability
in \cite{l-w-x,l-l-x}.

\subsection{Main results}

The main result in the present work establishes the previous conjecture
in the case when the $Y$ is an affine toric variety (recall that
a toric variety has $\Q-$Gorenstein singularities iff it it has klt
singularities \cite{dai}):
\begin{thm}
\label{thm:main intro}Let $Y$ be an affine toric variety which is
$\Q$-Gorenstein. Then the following is equivalent:
\begin{itemize}
\item $(Y,T,\xi)$ admits a conical Calabi\textendash Yau metric $\omega.$
\item $\xi$ is the unique minimizer of the volume functional on the space
of normalized Reeb vectors in the Lie algebra of the maximal torus
$T_{m}.$
\item $(Y,T,\xi)$ is K-stable. 
\end{itemize}
\end{thm}

In fact, the Calabi-Yau metric $\omega$ can be taken to be $T_{m}-$invariant
with locally bounded $T_{m}-$invariant potential $r^{2}$ (see Theorem
\ref{thm:toric CY iff vol min}) and hence, by local results \cite[Prop 4.1]{c-g-s},
$r^{2}$ is automatically continuous on $Y.$ 

By \cite{C-S} and \cite[Remark 2.27]{l-x} the implication ``K-stability$\implies$volume
minimization'' holds for any polarized affine variety $Y$ with klt
singularities, but the converse implication, resulting from the previous
theorem, is a special feature of the toric setting (but see \cite{l1,l-l-x}
for a generalized notion of volume minimization, applying to general
affine varieties, where the role of Reeb vectors is played by general
valuations). Combining the previous theorem with \cite[Thm 7.1]{C-S}
and \cite[Thm 4.1, Cor 4.2]{l-w-x} thus yields an analytic proof
of the following purely algebro-geometric result (recall that $T(\xi)$
denotes the minimal torus determined by $\xi)$:
\begin{cor}
\label{cor:of main thm intro}Let $Y$ be an $m$-dimensional affine
toric variety which is $\Q$-Gorenstein and denote by $T_{m}$ the
corresponding torus of maximal rank $m.$ If $(Y,T_{m},\xi)$ is K-stable
then $(X,T(\xi))$ is K-stable with respect to all weak special test
configurations and Ding polystable with respect to all $\Q$-Gorenstein
test configurations.
\end{cor}

The previous corollary is closely related to results in \cite{l-w-x}
comparing K-(semi-)stability with $T$-equivariant K-(semi-)stability,
shown using purely algebro-geometric techniques from the MMP.

The case of Theorem \ref{thm:main intro} when $y_{0}$ is an isolated
singularity was first shown in \cite{f-o-w,c-f-o} (using a method
of continuity, similar to the projective Fano manifold case in \cite{w-z}).
The extension to general toric affine varieties $Y$ was conjectured
in the last section of \cite{C-S} (see also \cite[page 357-358]{d-sII}
for relations to toric tangent cones). As discussed in \cite{C-S},
general toric affine varieties $Y$ are expected to appear as degenerations
of certain K-unstable non-toric affine varieties. See also \cite{l-l-x}
for a discussion about a more general conjecture concerning general
klt singularities $(Y,y_{0})$ (the Stable Degeneration Conjecture). 

As stressed in \cite{m-s-y0} toric affine varieties also play a prominent
role in the AdS/CFT correspondence (mainly for $m=3$ and $m=4)$,
since the corresponding supersymmetric gauge theories may be constructed
explicitly, using toric quivers, encoded by brane tilings and dimers
(see \cite{xiao} for an exposition aimed at mathematicians when $m=3$
and \cite{m-s2} for $m=4$). 
\begin{example}
Non-isolated toric singularities also play an important role in the
AdS/CFT correspondence. This is illustrated by the simple orbifold
case $Y_{0}=\C^{3}/\Z_{2},$ where $-1\in\Z_{2}$ acts as $(z_{1},z_{2},z_{3})\mapsto(z_{1},-z_{2},-z_{3})$
and thus $Y$ is singular over the complex line $(\C,0,0).$ The affine
variety $Y_{0}$ may be embedded as the hypersurface in $\C^{4}$
defined by $w_{2}w_{3}=w^{4}$ and corresponds to a gauge theory with
$\mathcal{N}=2$ supersymmetry. The non-isolated singularity locus
of $Y_{0}$ is ``resolved'' by the deformation of $Y$ defined by
the conifold $Y_{\epsilon},$ $\epsilon w_{1}^{2}+w_{2}w_{3}=w^{4},$
which corresponds to a gauge theory with merely $\mathcal{N}=1$ supersymmetry
(see \cite[pages 13-17]{k-w}). The AdS/CFT correspondence for general
affine (log terminal) singularities is discussed in \cite{m-p}.
\end{example}

The proof of Theorem \ref{thm:main intro} also yields a new variational
proof of the existence results in \cite{f-o-w,c-f-o}. The starting
point is the basic observation that the restriction of a toric conical
Calabi\textendash Yau metric $\omega_{Y}$ to the open orbit of $T_{\C}$
in $Y$ corresponds to the solution $f(x)$ of an inhomogeneous real
Monge\textendash Ampère equation on $\R^{m}$ with exponential right
hand side, using the standard fibration $T_{\C}\rightarrow\R^{m}$,
where $\R^{m}$ is identified with the Lie algebra of $T.$ The image
of the gradient of $f$ is prescribed to be the convex convex $\mathcal{C}^{*}$
defined by the moment polytope of the toric variety $Y.$ The existence
of the metric $\omega_{Y}$ and its regularity on the open orbit of
$T_{\C}$ in $Y$, as well as an a priori $L^{\infty}$-bound, then
follows from a general result about such Monge\textendash Ampère equations
on $\R^{m},$ which may be of independent interest (see Theorem \ref{thm:MA for convex cone}).
The latter result is obtained by applying \cite[Thm 1.1]{be-be},
concerning the second boundary value problem for the Monge\textendash Ampère
equation with an exponential non-linearity associated to a convex
body $P,$ to the $(m-1)-$dimensional linear space $\R^{m}/\R\xi$
and a compact (possible irrational) polytope obtained by intersecting
the convex cone $\mathcal{C^{*}}$ with a hyperplane determined by
$\xi.$ This leads, in particular, to a variational construction of
the conical Calabi\textendash Yau metric $\omega_{Y}.$ Finally, in
order to show that $\omega_{Y}$ is smooth on all of $Y_{reg}$ we
employ a transversal generalization of the Laplacian a priori estimates
in \cite[Appendix B]{bbegz}, by constructing an appropriate barrier
and exploiting a generalization of the $L^{\infty}$-estimate in Theorem
\ref{thm:MA for convex cone}. 

As will be shown in \cite{berm13} the toric conical Calabi-Yau metric
appearing in Theorem \ref{thm:main intro} can be constructed probabilistically
by sampling certain explicitely defined random point processes (which
are tropicalizations of the canonical point processes on Fano varieties
and Sasaki varieties introduced in \cite{berm8 comma 5} and \cite{b-c-p},
respectively). 

\subsection{Acknowledgments}

I am grateful to Mingchen Xia for discussions and many helpful comments.
Also thanks to Daniel Persson for stimulating discussions about relations
to the AdS/CFT correspondence. This work was supported by grants from
the KAW foundation, the Göran Gustafsson foundation and the Swedish
Research Council. 

\section{\label{sec:Convex-cones-and}Convex cones and real Monge\textendash Ampère
equations}

According to the classical Jörgens\textendash Calabi\textendash Pogorelov
theorem \cite{po} a smooth convex function $f$ on Euclidean space
$\R^{m}$ solves the Monge\textendash Ampère equation 
\[
\det\nabla^{2}f=g
\]
 with $g$ constant iff $f$ is quadratic. Here we will be concerned
with the case when the right hand side $g$ is of the form
\[
g(x)=Ce^{\left\langle l,x\right\rangle }
\]
 for a given \emph{non-zero} vector $l$ and some non-zero constant
$C$ (which after a scaling could as well have be taken to be equal
to one). In this case it turns out that the solution space is, in
general, infinite dimensional (if $m\geq3)$. Some extra conditions
thus need to be imposed to get a finite dimensional space of solutions.
Motivated by the setting of toric conical Calabi\textendash Yau metrics
we will impose that the gradient image 

\[
\overline{(\nabla f)(\R^{m})}=\mathcal{C}^{*},
\]
for a given \emph{proper convex cone} $\mathcal{C}^{*}$ in $\R^{m},$
i.e. $\mathcal{C}$ is a closed convex cone of dimension $m$ such
that $\mathcal{C}$ does not contain a line. Moreover, we demand that
the convex solution $f$ is strictly positive and satisfies the following
homogeneity property: there exists a vector $\xi\in\R^{m}$ such $f$
\emph{is one-homogeneous with respect to $\xi,$} i.e. for any $t\in\R$
\begin{equation}
f(x+t\xi)=e^{t}f(x).\label{eq:one homog wrt xi}
\end{equation}
Recall that a proper convex cone is a closed convex cone which does
not contain an entire line. 
\begin{example}
When $l=(1,...,1)$ and $\mathcal{C}^{*}$ is the ``non-negative
quadrant'' a solution (with $C=1)$ is obtained by setting 
\[
f(x)=\sum_{i\leq m}e^{x_{i}}
\]
(and any translation of $f$ yields a new solution). 
\end{example}

Denote by $\mathcal{C}$ the dual convex cone of the convex cone $\mathcal{C}^{*},$
i.e. the set of all $x\in\R^{m}$ such that the corresponding linear
function $\left\langle x,\cdot\right\rangle $ is non-negative on
$\mathcal{C}^{*}.$ Then $\mathcal{C}^{*}$ is a proper convex cone
iff $\mathcal{C}^{*}$ is (using that, conversely, the dual of $\mathcal{C}$
is $\mathcal{C}^{*}).$ An element $\xi$ in the interior of $\mathcal{C}$
will be called a \emph{Reeb vector}. Since the determinant of an $m\times m$
matrix is homogeneous of degree $m$ a necessary condition for the
existence of a solution $f,$ which is one-homogeneous with respect
to $\xi$, is that $\xi$ is\emph{ $l$-normalized} in the sense that
\[
\left\langle l,\xi\right\rangle =m.
\]
According to the next theorem a sufficient condition is that $\xi$
minimizes the \emph{volume} $V(\xi)$ defined as 
\begin{equation}
V(\xi):=\text{Vol \ensuremath{(\mathcal{C}^{*}\cap\left\{ \left\langle \xi,\cdot\right\rangle \leq1\right\} \in]0,\infty[} }\label{eq:def of V convex setting}
\end{equation}
computed with respect to Lebesgue measure on $\R^{m}.$ 
\begin{thm}
\label{thm:MA for convex cone}Let $\mathcal{C}^{*}$ be a proper
convex cone in $\R^{m},$ for $m\geq2,$ and $l$ a non-zero vector
in the interior of $\mathcal{C}^{*}$. Given a Reeb vector $\xi$
the following is equivalent: 
\begin{itemize}
\item There exists a smooth strictly positive convex function $f:\R^{m}\rightarrow\R$
which is one-homogeneous with respect to $\xi$ and satisfies
\[
\det\nabla^{2}f=Ce^{\left\langle l,x\right\rangle },\,\,\,\overline{(\nabla f)(\R^{m})}=\mathcal{C}^{*}
\]
for some non-zero constant $C.$
\item $\xi$ minimizes the volume $V(\xi)$ among all $l$-normalized Reeb
vectors. 
\end{itemize}
Moreover, there exists a unique such minimizer $\xi$ and the corresponding
solution $f$ is uniquely determined modulo translations, i.e. the
additive group $\R^{m}$ acts transitively on the solution space. 
\end{thm}

In fact, the proof yields the following property of $f$ which is
stronger than the gradient property above: there exists a constant
$C$ such that 
\begin{equation}
f(x)\leq Ce^{\phi_{P_{\xi}}(x)},\label{eq:bound on f}
\end{equation}
 where $\phi_{P}(x)$ denotes the support function of a given convex
set $P\Subset\R^{n},$ i.e. 
\begin{equation}
\phi_{P}(x)=\sup_{p\in P}\left\langle x,p\right\rangle .\label{eq:def of support f}
\end{equation}
and $P_{\xi}$ is the convex bounded set defined by
\begin{equation}
P_{\xi}:=\mathcal{C}^{*}\cap\left\{ \left\langle \xi,\cdot\right\rangle =1\right\} .\label{eq:def of P xi}
\end{equation}
Note that it follows directly from the definition of $P_{\xi},$ that
$e^{\phi_{P_{\xi}}(x)}$ is one-homogeneous with respect to $\xi.$

\subsection{Proof of Theorem \ref{thm:MA for convex cone}}

It will be convenient to reformulate the theorem in a linearly invariant
manner. Thus starting with the real vector space $W:=\R^{m}$ we let
$\mathcal{C}^{*}$ be a proper convex cone in the space $W^{*}$of
linear functions on $W.$ We identify $l$ with a non-zero element
in $W^{*}$ and $\mathcal{C}$ with a cone in $W.$ In particular,
a Reeb vector $\xi$ defines an element in $W.$ A function $f:W\rightarrow\R$
as in Theorem \ref{thm:MA for convex cone} then satisfies 
\begin{equation}
MA(f)=Ce^{l}dx,\,\,\,\,\overline{(df)(W)}=\mathcal{C}^{*},\label{eq:MA eq for f in pf}
\end{equation}
 where $dx$ denotes a Lebesgue measure on $W$ and $MA(f)$ is the
real Monge-{}-Ampère measure,
\begin{equation}
MA(f):=\det(\nabla_{x}^{2}f)dx\label{eq:def of MA f}
\end{equation}
(the equation is independent of the choice of linear coordinates $x$
on $W,$ up to rescaling the constant $C).$ Assume that $\xi$ is
$l$-normalized. Writing $f=e^{\Phi}$ the homogeneity property \ref{eq:one homog wrt xi}
for $f$ holds iff the function $\Phi$ is $\xi$-equivariant, i.e.
$\Phi(x+t\xi)=\Phi(x)+t$ for any $t\in\R$. Since $d(e^{\Phi})=e^{\Phi}d\Phi$
this equivalently means that 
\begin{equation}
\overline{(d\Phi)(W)}=P_{\xi},\label{eq:gradient cond for Phi}
\end{equation}
 where $P_{\xi}$ is defined in formula \ref{eq:def of P xi}. Next,
setting

\[
\phi:=\Phi-l/m,
\]
 gives a $\xi$-invariant function, i.e $\phi(x+t\xi)=\phi(x)$. Equivalently,
one may regard $\phi$ as a function defined on $W/\R\xi$. The condition
in formula \ref{eq:gradient cond for Phi} is equivalent to 
\begin{equation}
\overline{(d\phi)(W)}=P_{\xi}-l/m.\label{eq:gradient cond for phi}
\end{equation}
Now consider the linear function 
\[
t:=l/m\in W^{*}
\]
and take $m-1$ linearly independent linear functions $s_{1},...,s_{m-1}$
on $W$ such that $s_{i}$ are $\xi$-invariant, i.e. $\left\langle s_{i},\xi\right\rangle =0.$
We then obtain an invertible linear map 
\begin{equation}
W\rightarrow\R^{m},\,\,\,x\mapsto(s_{1}(x),\ldots,s_{m-1}(x),t(x))\label{eq:x maps to}
\end{equation}
The Monge-{}-Ampère equation \ref{eq:MA eq for f in pf} for $e^{\Phi}$
can thus be expressed as 
\[
\left(\det\nabla_{s}^{2}\right)\left(e^{t+\phi(s)}\right)=ae^{mt},
\]
 for some non-zero constant $a$ (coming from the Jacobian of the
linear map \ref{eq:x maps to}). By a direct computation this equivalently
means that on $W/\R\xi\cong\R^{m-1}$ with coordinates $s$ we have
(after perhaps rescaling $s$)
\begin{equation}
\left(\det\nabla_{s}^{2}\right)\left(\phi(s)\right)=e^{-m\phi(s)}\label{eq:MA eq for phi of X}
\end{equation}
 This means that a $\xi$-homogeneous smooth convex function $e^{\Phi}$
solves the equation \ref{eq:MA eq for f in pf} subject to \ref{eq:MA eq for phi of X}
iff the smooth convex function $\phi$ on $W/\R\xi$ solves the equation
\ref{eq:MA eq for phi of X} subject to 
\begin{equation}
\overline{(\nabla\phi)(W)}=Q_{\xi}\subset\xi^{\perp},\label{eq:def of Q}
\end{equation}
 where $\xi^{\perp}\subset W^{*}$ is the subet of linear functions
on $W$ that vanishes on $\xi$, $Q_{\xi}$ is the convex body defined
as the $P_{\xi}-l/m$ regarded as a subset of $\xi^{\perp}$. By \cite[Thm 1.1]{be-be}
there exists a solution $\phi$ iff the barycenter of $Q_{\xi}$ is
equal to the origin in $\xi^{\perp}.$ Equivalently, this condition
means that the barycenter of $P_{\xi}$ is equal to $l/m.$ Moreover,
any other solution $\phi$ is of the form 
\[
\phi(s+a)+c
\]
 for some $a\in\xi^{\perp}$ and $c\in\R.$ Furthermore, by the general
$L^{\infty}$-estimate in \cite[Thm 1.1]{be-be}, detailed in the
following lemma, a solution $\phi$ satisfies the following global
bound 
\[
\sup_{W/\R\xi}|\phi-\phi_{0}|\leq C,\,\,\,\phi_{0}(x):=\sup_{p\in Q_{\xi}}\left\langle x,p\right\rangle 
\]
for some constant $C,$ using that 
\begin{equation}
\mu:=e^{-m\phi(s)}ds\leq Ae^{-|s|/A}ds\label{eq:exp decay}
\end{equation}
 for some constant $A,$ since $0$ is an interior point of $(\nabla\phi)(W/\R\xi).$ 
\begin{lem}
\label{lem:L infty estimate in convex settting}Let $\phi$ be a convex
function on $\R^{n}$ such that 
\[
MA(\phi)=\mu,\,\,\,\overline{(\partial\phi)(\R^{n})}=P
\]
 for a convex body $P\subset\R^{n}$ and a measure $\mu$ on $\R^{n},$
where $MA(\phi)$ denotes the Monge\textendash Ampère measure of $\phi$
(defined by formula \ref{eq:def of MA f} when $\phi$ has a bounded
Hessian). Set 
\[
\phi_{P}(x):=\sup_{p\in P}\left\langle x,p\right\rangle 
\]
and assume that $\phi$ is normalized so that $\sup_{\R^{n}}(\phi-\phi_{P})=0.$
Given $q>n$ the following inequality holds
\[
\sup_{\R^{n}}|\phi-\phi_{P}|\leq\frac{d(P)}{V(P)}\int_{\R^{n}}|x|\mu+C_{n,q}\frac{d(P)^{\left(1+n(1-1/q)\right)}}{V(P)}\left(\int_{\R^{n}}|x|^{q}\mu\right)^{1/q},
\]
 where $d(P)$ denotes the diameter of $P,$ $V(P)$ its volume and
the constant $C_{n,q}$ only depends on $n$ and $q.$ 
\end{lem}

\begin{proof}
Following the argument in the proof of \cite[Prop 2.2]{be-be} we
denote by $v:=\phi^{*}$ is the Legendre transform of $\phi,$ which
defines a convex function on the interior of $P.$ By the Sobolev
inequality for the embedding $W^{1,q}(P)\Subset L^{\infty}(P)$ we
have, since the interior of $P$ is a bounded convex domain,
\[
\sup_{P}|v|\leq\frac{1}{V(P)}\int_{P}|v(y)|dy+C_{n,q}\frac{d(P)^{\left(1+n(1-1/q)\right)}}{V(P)}\left(\int_{P}|\nabla v(y)|^{q}dy\right)^{1/q}
\]
(see \cite[Lemma 1.7.3]{da} or \cite[Thm 4.4]{m-t-s-o}). In general,
as explained in \cite[Prop 2.2]{be-be}, $\inf_{P}v=-\sup(\phi-\phi_{P})$
and hence, by assumption, $\inf_{P}v=0.$ Assume that the infimum
is attained at $y_{0}\in P,$ i.e. $v(y_{0})=0.$ By convexity $|v(y)|=v(y)-v(y_{0})\leq\nabla v(y)\cdot(y-y_{0}).$
Thus the Cauchy-Schwartz inequality yields 
\[
\text{\ensuremath{\int_{P}|v(y)|dy\leq d(P)}}\int_{P}|\nabla v|dy.
\]
Finally, the proof is concluded by observing that $\int|\nabla v(y)|^{\alpha}dy=\int|x|^{\alpha}MA(\phi)$
for any $\alpha>0$ and $\sup_{P}|v|=\sup_{\R^{n}}|\phi-\phi_{P}|.$
\end{proof}
The proof of Theorem \ref{thm:MA for convex cone} is now concluded
by invoking the following 
\begin{prop}
\label{prop:bary}A Reeb vector $\xi$ minimizes the volume among
all $l$-normalized Reeb vectors iff the barycenter $b_{P_{\xi}}$
of $P_{\xi},$ i.e. the element of $P_{\xi}$ defined by
\[
b_{P_{\xi}}:=\int_{p\in P_{\xi}}pd\lambda_{m-1}/\int_{p\in P_{\xi}}d\lambda_{m-1},
\]
 (where $d\lambda_{m-1}$ denotes any choice of Lebesgue measure on
the $(m-1)$-dimensional affine space $\left\langle \xi,\cdot\right\rangle =1)$
is given by
\[
b_{P_{\xi}}=l/m
\]
Moreover, there always exists a unique minimizer $\xi.$ 
\end{prop}

To prove the proposition first observe that $V(\xi)\rightarrow\infty$
as $\xi$ approaches a non-zero point $\xi_{0}$ in the boundary of
$\mathcal{C}.$ Indeed, there exist non-zero elements $p_{0},p_{1}$
in the interior of $\mathcal{C}^{*}$ such that $\left\langle \xi_{0},p_{1}\right\rangle =1$
and $\left\langle \xi_{0},p_{0}\right\rangle =0.$ As a consequence,
the corresponding convex set $P_{\xi_{0}}$ is unbounded (since it
contains $p_{1}+cp_{0}$ for any $c>0)$ and hence $V(\xi_{0})=\infty.$
Thus, the restriction of $V$ to the convex bounded set $\mathcal{C^{*}}\cap\left\{ \left\langle l,\cdot\right\rangle =m\right\} $
admits a minimizer $\xi_{*}$ in the interior, i.e. a minimizing Reeb
vector field. 
\begin{lem}
\label{lem:The function log V }The function $\log V(\xi)$ is smooth
and strictly convex on the interior of $\mathcal{C}$ and its differential
is given by
\[
d\log V_{|\xi}=-mb_{P_{\xi}}
\]
\end{lem}

\begin{proof}
First observe that
\begin{equation}
V(\xi)=m!\int_{\mathcal{C}^{*}}e^{-\left\langle \xi,p\right\rangle }dp.\label{eq:V as Laplac transf}
\end{equation}
 Indeed, setting $s:=\left\langle \xi,p\right\rangle $ we have 
\[
\int_{\mathcal{C}^{*}}e^{-\left\langle \xi,p\right\rangle }dp=\int e^{-s}s_{*}(1_{\mathcal{C}^{*}}dp)=\int_{]0,\infty[}e^{-s}\frac{d}{ds}V(s)ds,\,\,\,V(s):=\int_{\{\xi<s\}}1_{\mathcal{C}}dp
\]
But since $V$ is homogeneous of degree $m$, we have $V(s)=s^{m}V(\xi)$
and hence formula \ref{eq:V as Laplac transf} follows from computing
\[
m\int_{]0,\infty[}e^{-s}s^{m-1}ds=m!
\]
The first statement of the lemma then follows from standard convex
analysis. Finally, we have
\[
-d_{\xi}\log\int_{\mathcal{C}^{*}}e^{-\left\langle \xi,p\right\rangle }dp=\frac{\int_{\mathcal{C}^{*}}pe^{-\left\langle \xi,p\right\rangle }dp}{\int_{\mathcal{C}^{*}}e^{-\left\langle \xi,p\right\rangle }dp}=mb_{P_{\xi}},
\]
 where the last equality follows from setting $s=\left\langle \xi,p\right\rangle $
and changing the order of integration to get 
\[
\int_{\mathcal{C}^{*}}pe^{-\left\langle \xi,p\right\rangle }dp=\left(\int_{\mathcal{C}^{*}\cap\{\left\langle \xi,p\right\rangle =s\}}pd\lambda_{n}\right)e^{-s}ds=\left(\int_{\mathcal{C}^{*}\cap\{\left\langle \xi,p\right\rangle =1\}}pd\lambda_{m-1}\right)\int_{0}^{\infty}s^{m}e^{-s}ds
\]
 and 
\[
\int_{\mathcal{C}^{*}}e^{-\left\langle \xi,p\right\rangle }dp=\left(\int_{\mathcal{C}^{*}\cap\{\left\langle \xi,p\right\rangle =s\}}d\lambda_{m-1}\right)e^{-s}ds=\left(\int_{\mathcal{C}^{*}\cap\{\left\langle \xi,p\right\rangle =1\}}d\lambda_{m-1}\right)\int_{0}^{\infty}s^{n}e^{-s}ds.
\]
Let now $\xi$ be the unique minimizer in question, Since, as explained
above, $\xi$ is an interior point of the convex set $\mathcal{C}\cap\left\{ \left\langle l,\cdot\right\rangle =m\right\} $
it follows from the previous lemma that
\[
b_{\xi}=cl
\]
 for a non-zero constant $c.$ Finally, since $1=\left\langle \xi,b_{\xi}\right\rangle =\left\langle \xi,l\right\rangle /m$
it must be that $c=1/m$, which concludes the proof of Prop \ref{prop:bary}. 
\end{proof}

\subsection{\label{subsec:Variational-construction-of}Variational construction
of the solution $f$}

Given a proper convex cone $\mathcal{C}^{*}$ in $(\R^{m})^{*}$ and
an $l$-normalized Reeb vector $\xi$ denote by $\mathcal{E_{\xi}}(\R^{m})$
the space of all convex functions $\Phi$ on $\R^{m}$ such that $\Phi$
is equivariant with respect to $\xi$ and $\overline{(\partial\Phi)(\R^{m})}=P_{\xi},$
where $P_{\xi}$ is the convex body defined by formula \ref{eq:def of P xi}
and $\partial\Phi$ denotes the sub-gradient of $\Phi,$ viewed as
a multivalued function. Denote by $\mathcal{E}_{\xi}^{1}(\R^{m})$
the subspace of $\mathcal{E_{\xi}}(\R^{m})$ consisting of all $\Phi$
such that $\Phi^{*}\in L^{1}(P_{\xi}),$ where $\Phi^{*}$ is the
Legendre\textendash Fenchel transform of $\Phi$, i.e. the convex
function on $P_{\xi}$ defined by 
\[
\Phi^{*}(p):=\sup_{x\in\R^{m}}\left\langle x,p\right\rangle -\Phi(x).
\]
 Consider the following functionals on $\mathcal{E}_{\xi}^{1}(\R^{m}):$ 

\begin{equation}
\mathcal{D}(\Phi):=-\frac{1}{m}\log\int_{\R^{m}/\R\xi}e^{-m\Phi}e^{\left\langle l,x\right\rangle }i_{\xi}dx-\mathcal{E}(\Phi)\label{eq:Ding in convex setting}
\end{equation}
 where $i_{\xi}dx$ denotes the $(n-1)$-form on $\R^{m}$ obtained
by contracting $dx$ with $\xi,$ which yields a well-defined measure
on $\R^{m}/\R\xi$ and 
\[
\mathcal{E}(\Phi):=-\frac{1}{V(P_{\xi},d\lambda_{m-1})}\int_{P_{\xi}}\Phi^{*}d\lambda_{m-1},
\]
 where $d\lambda_{m-1}$ denotes any choice of Lebesgue measure on
the hyperplane \{$\xi=1\}$ and $V(P_{\xi},d\lambda_{m-1})$ denotes
the volume of $P_{\xi}$ with respect to $d\lambda_{m-1}.$
\begin{thm}
\label{thm:variational convex}The following holds. 
\begin{itemize}
\item The functional $\mathcal{D}$ is bounded from below on $\mathcal{E}_{\xi}^{1}(\R^{m})$
iff $\xi$ is the unique minimizer of the volume $V(\xi)$ on the
Reeb cone. 
\item $\Phi$ minimizes $\mathcal{D}$ on $\mathcal{E}_{\xi}^{1}(\R^{m})$
iff $e^{\Phi}$ is a solution to the equation in Theorem \ref{thm:MA for convex cone}.
\end{itemize}
Moreover, if $\Phi_{j}$ is a minimizing sequence for $\mathcal{D}$
in $\mathcal{E}_{\xi}^{1}(\R^{m}),$ i.e. 
\[
\lim_{j\rightarrow\infty}\mathcal{D}(\Phi_{j})=\inf_{\mathcal{E}_{\xi}^{1}(\R^{m})}\mathcal{D>-\infty},
\]
 then there exist sequences $a_{j}\in\R^{m}$ such that the translated
sequence $\Phi_{j}(a_{j}+x)$ converges in $C_{loc}^{0}(\R^{m})$
to $\Phi\in\mathcal{E}_{\xi}^{1}(\R^{m}),$ where $e^{\Phi}$ is a
solution to the equation in Theorem \ref{thm:MA for convex cone}.
\end{thm}

\begin{proof}
In terms of the coordinates $(s,t)$ on $\R^{m},$ appearing in formula
\ref{eq:x maps to}, we can write $\Phi=\phi(s)+t,$ where $\phi(s)$
is a convex function on $\R^{m-1}$ such $\overline{(\partial\phi)(\R^{m-1})}=Q_{\xi}$
and $\phi^{*}\in L^{1}(Q_{\xi}).$ Denote by $\mathcal{E}^{1}(\R^{m-1})$
the space of all such functions $\phi.$ Then 
\[
\mathcal{D}(\Phi)=D(\phi):=-\frac{1}{m}\log\int_{\R^{n}}e^{-m\phi(s)}ds+\frac{1}{V(Q_{\xi},dq)}\int_{Q_{\xi}\Subset\R^{n}}\phi^{*}(q)dq.
\]
Recall that, as shown in the proof of Theorem \ref{thm:MA for convex cone},
$\xi$ minimizes $V(\xi)$ iff $0$ is the barycenter of the convex
body $Q_{\xi}.$ By \cite[Thm 2.16]{be-be} the latter condition equivalently
means that $D$ is bounded from below on $\mathcal{E}^{1}(\R^{m-1}).$
Moreover, as shown in Section 2.8 and Section 2.9 of \cite{be-be}
$\phi$ minimizes $D$ on $\mathcal{E}^{1}(\R^{m-1})$ iff $\phi$
is a smooth solution of the equation \ref{eq:MA eq for phi of X},
which, as shown in the proof of Theorem \ref{thm:MA for convex cone},
equivalently means that $e^{\Phi}$ is a solution to the equation
in Theorem \ref{thm:MA for convex cone}. Finally, the convergence
of a $\mathcal{D}-$minimizing sequence $\Phi_{j}$ follows from the
convergence of a $D$-minimizing sequence $\phi_{j}$ established
in the proof of \cite[Thm 1.1]{be-be}.

In fact, setting 
\[
J(\Phi):=-\mathcal{E}(\Phi)+\sup_{\R^{m}}(\Phi-\Phi_{P_{\xi}})
\]
the proof of the previous theorem shows (thanks to the results in
\cite{be-be} alluded to above) that $\xi$ minimizes the volume function
on the Reeb cone iff $\mathcal{D}$ is coercive, modulo translations
or more precisely: iff there exists a constant$C>0$ such that

\[
\mathcal{D}(\Phi)\geq\frac{1}{C}\inf_{a\in\R^{m}}J(a^{*}\Phi)-C,
\]
 where $a^{*}\Phi(x):=\Phi(x+a)$ is the action on $\Phi$ by a translation
of $\R^{m}.$ 
\end{proof}

\section{Toric varieties and Conical Calabi\textendash Yau metrics}

\subsection{The toric setup}

Fixing the rank $m$ we will denote by $T_{\C}$ the complex torus
$(\C^{*})^{m},$ endowed with its standard group structure and by
$T$ the corresponding compact torus. The group of characters on $T_{\C}$
and the group of 1-parameters subgroups of $T_{\C}$ are denoted,
as usual, by $M$ and $N,$ respectively: 
\[
M:=\text{Hom\ensuremath{(T_{\C},\C^{*}),\,\,\,N:=}\text{Hom\ensuremath{(\C^{*},T_{\C})} }}
\]
Composing the corresponding homomorphisms yields a character on $\C^{*},$
which defines a non-degenerate pairing between $M$ and $N.$ The
real vector spaces $N\otimes\R$ and $M\otimes\R$ are naturally identified
with the Lie algebra $\mathfrak{t}$ of $T$ and its dual $\mathfrak{t}^{*}$
respectively. 

Let now $Y$ be an $m-$dimensional\emph{ normal affine toric variety},
i.e. $Y$ is a normal affine variety endowed with a faithful action
by the complex torus $T_{\C}$ with an open dense orbit. An extensive
exposition of general toric varieties is given in the book \cite{c-l-s}.
Anyway, we will recall the basic properties of affine toric varieties
that we shall need.

We will identify $T_{\C}$ with the corresponding open orbit in $Y.$
Decompose the ring of holomorphic functions $\mathcal{R}(Y)$ on $Y$
with respect to the induced $T_{\C}-$action on $\mathcal{R}(Y)$:
\[
\mathcal{R}(Y)=\oplus_{p\in G}\C F_{p},
\]
 where $G$ denotes the corresponding set of characters in $M.$ Restricting
$F_{p}$ to the dense $T_{\C}-$orbit, endowed with its standard complex
coordinates $(z_{1},...,z_{m}),$ we can identify a character $F_{p}$
with a multinomial:
\[
F_{p}:=z^{p}:=z_{1}^{p_{1}}\cdots z_{m}^{p_{m}},
\]
 where $G$ has been identified with $\Z^{m}.$ Since $\mathcal{R}$
is a finitely generated ring the set $G$ forms a commutative semi-group
in $M,$ which is finitely generated (i.e. $G$ is an affine semi-group).
Denote by $\mathcal{C}^{*}$ the convex hull of the set $G:$ 
\[
\mathcal{C}^{*}:=\text{Conv\ensuremath{(G)\subset M\otimes\R(=\mathfrak{t}^{*})}}
\]
Then $\mathcal{C}$ is a proper convex cone, in the sense of Section
\ref{sec:Convex-cones-and}, which is polyhedral and the action of
$T$ on $Y$ is good, in the sense of Section \ref{subsec:Background}
(see \cite[Thm 1.3.5]{c-l-s} and \cite[Prop 1.2.12]{c-l-s}; proper
convex cones are called strongly convex cones in \cite{c-l-s}).

Denote by $\mathcal{C}$ the dual convex cone defined by all elements
in the linear dual of $M\otimes\R$ which are non-negative on $\mathcal{C}^{*}.$
Given an element $\xi$ in $N\otimes\R$ its weights are defined by
the real numbers 
\[
\lambda_{\xi}(p):=\frac{1}{i}\mathcal{L}_{\xi}F_{p}=\left\langle \xi,p\right\rangle ,\,\,\,p\in G
\]
Hence, $\xi$ is a Reeb vector (in the terminology recalled in Section
\ref{subsec:Background}) iff $\xi$ is an interior point of the cone
$\mathcal{C}.$ Accordingly, the interior of $\mathcal{C}$ will be
called the \emph{Reeb cone}. Denote by $\xi_{1},...,\xi_{d}$ the
non-zero primitive vectors in the lattice $N$ such that 
\[
\mathcal{C}^{*}=\left\langle \xi_{i},\cdot\right\rangle \geq0,\,\,\,i=1,..,d
\]
Equivalently, this means that $\mathcal{C}$ is the convex hull of
$0,\xi_{1},...,\xi_{d}.$ 

Given a Reeb vector $\xi$ and a positive number $C$ we let $N(\xi,C)$
be the number of lattice points $p$ in $\mathcal{C}^{*}$satisfying
$\left\langle \xi,p\right\rangle \leq C.$ Fixing an enumeration of
the points in question we have the following basic
\begin{prop}
\label{prop:toric embedding}Given a Reeb vector $\xi$ and $C>0.$
For $C$ sufficiently large the map 
\[
(\C^{*})^{m}\rightarrow\C^{N},\,\,\,z\mapsto Z:=\left(z^{p_{1}},...,z^{p_{N}}\right)
\]
 with $N:=N(\xi,C)$ is injective and $T_{\C}-$equivariant and its
closure is isomorphic to the affine variety $Y$ corresponding to
the convex cone $\mathcal{C}^{*}.$ 
\end{prop}

\begin{proof}
Since the group $G$ is finitely generated we have that $\mathcal{C}^{*}=\N(\{p_{1},...,p_{r}\}$
for a finite number of $p_{i}\in\mathcal{C}^{*}.$ For $C\geq\min\left\langle \xi,p_{i}\right\rangle $
we thus get $\mathcal{C}^{*}=\N(\{\mathcal{C}^{*}\cap\left\langle \xi,\cdot\right\rangle \leq C\}$
and hence the proposition follows from \cite[Prop 1.1.14]{c-l-s}. 
\end{proof}
Given a Reeb vector $\xi$ exponentiating the vector field on $T_{\C}$
defined by $-J\xi$ yields an $\R^{*}-$action on $Y.$ Accordingly,
the embedding in the previous proposition is $\R^{*}-$ equivariant
when $\C^{N}$ is endowed with the linear action 
\[
(c,Z_{i})\mapsto c^{\left\langle \xi,p_{i}\right\rangle }Z_{i},\,\,\,\R^{*}\times\C^{N}\rightarrow\C^{N}
\]
(the space $\C^{N}$ endowed with such an $\R^{*}-$action is called\emph{
weighted $\C^{N}).$} We also recall the following
\begin{prop}
\label{prop:Q Gore}An affine toric variety $Y$ is $\Q-$Gorenstein
iff there exists an element $l\in M\otimes\Q$ such that 
\[
\left\langle \xi_{i},l\right\rangle =1
\]
In particular, $l$ is an interior point of the cone $\mathcal{C}^{*}.$
\end{prop}

\begin{proof}
This is well-known \cite{dai}, but for the convenience of the reader
we outline a proof. First assume that $Y$ is $\Q-$Gorenstein (which
is the direction that we shall need). Then it admits a (manyvalued)
holomorphic top form $\Omega$ (defined and non-vanishing on the regular
locus) which is equivariant with respect to the action of the torus
$T_{\C}.$ Hence, the restriction of $\Omega$ to $T_{\C}\subset Y$
may be expressed as 
\begin{equation}
\Omega=\chi\Omega_{0},\,\,\,\Omega_{0}:=z_{1}^{-1}dz_{1}\wedge...\wedge z_{m}^{-1}dz_{m}\label{eq:Omega in terms of Omega not}
\end{equation}
for some many valued character $\chi,$ i.e. $\chi^{r}$ is a well-defined
character for some integer $r.$ The holomorphicity condition and
the non-vanishing condition implies that $\nu_{F_{i}}(\Omega)=0,$
where, $\nu_{F}$ denotes the order of vanishing along a given prime
divisor on $Y$ and $F_{i}$ denote the toric prime divisors whose
union give $Y-T_{\C}.$ It follows readily from its definition that
$\Omega_{0}$ defines a rational (meromorphic) top form on $Y$ such
that $\nu_{F_{i}}(\Omega_{0})=-1.$ Hence, $\nu_{F_{i}}(\Omega)=0$
iff $\nu_{F_{i}}(\chi)=1.$ But writing $\chi^{r}=z^{p}$ for some
$p\in\mathcal{C}\cap M$ this means that $\left\langle \xi_{i},p/r\right\rangle =1$
(by the ``orbit-cone correspondence'' \cite{c-l-s}). Hence, we
can take $l:=p/r.$ Conversely, if an element $l$ as in the proposition
exists we can simply define $\Omega$ by formula \ref{eq:Omega in terms of Omega not}.
Then the previous argument shows that $\nu_{F_{i}}(\Omega)=0,$ which
shows that $\Omega$ defines a non-vanishing equivariant (manyvalued)
holomorphic top form on $Y$ with positive weight under $\xi.$ By
general facts, this implies that $Y$ is $\Q-$Gorenstein. 
\end{proof}

\subsubsection{Conical Kähler potentials $r^{2}$ and $\xi-$equivariant functions
$\Phi$}

A function $r$ on $Y$ will be said to be\emph{ radial with respect
to $\xi$ }if it is non-negative homogeneous of degree one under the
corresponding $\R^{*}-$action and $\xi-$invariant. A function $f$
on $Y$ will be called \emph{a conical Kähler potential with respect
to $\xi$ }if it may be expressed as $f=r^{2}$ for a $\xi-$radial
function $r$ on $Y$ and $f$ is smooth and strictly plurisubharmonic
(psh) on $Y^{*}$ and continuous on $Y$ (i.e. $f$ is the restriction
of a smooth strictly plurisubharmonic function on some local embedding
of $Y^{*}$ into $\C^{M}).$ 
\begin{lem}
\label{lem:family of smooth conical}Let $Y$ be an affine variety
endowed with the action of a good torus $T$. For any given Reeb vector
$\xi$ there exists a conical Kähler potential $f_{\xi}:=r_{\xi}^{2}$
on $Y$ which is $T-$invariant. Moreover, $f_{\xi}$ may be chosen
so that the set $\{f_{\xi}=1\}$ is independent of $\xi.$ 
\end{lem}

\begin{proof}
In the case when $Y-\{y_{0}\}$ is smooth this is the content of \cite[Lemma 2.2]{h-s}.
In particular, when $Y$ is weighted $\C^{N}$ we can define $f_{\xi}$
so that $\{f_{\xi}=1\}$ is the unit-sphere. The general case when
$Y$ is singular then follows by embedding $Y$ equivariantly into
weighted $\C^{N}$ (as in Prop \ref{prop:toric embedding}) and restricting
$f_{\xi}$ on $\C^{N}$ to the image of $Y$ (under this embedding
the action of $T$ on $Y$ is intertwined by the linear action of
a good torus on $\C^{N}).$
\end{proof}
Note that if $r^{2}$ is a conical Kähler potential, then the function
$\Phi:=\log r^{2}$ is\emph{ $\xi-$equivariant }in the following
sense:
\begin{equation}
\mathcal{L}_{\xi}\Phi=0,\,\,\mathcal{L}_{-J\xi}\Phi=2.\label{eq:xi equi}
\end{equation}

\subsection{The YTD conjecture for toric affine varities}

In order to prove Theorem \ref{thm:main intro}, stated in the introduction,
it will be enough to prove the following
\begin{thm}
\label{thm:toric CY iff vol min}Let $Y$ be an affine toric variety
which is $\Q-$Gorenstein and denote by $T_{m}$ the corresponding
maximal torus and $\xi$ a Reeb vector in the Lie algebra of $T_{m}.$
Then the following is equivalent:
\begin{itemize}
\item $(Y,T_{m},\xi)$ admits a conical Calabi-Yau metric, which is $T_{m}-$invariant
(and uniquely determined modulo the action of $T_{\C})$
\item $\xi$ is the unique minimizer of the volume functional on the space
of normalized Reeb vectors
\end{itemize}
\end{thm}

Indeed, if there exists a conical Calabi-Yau metric associated to
a Reeb vector $\xi,$ then it follows from \cite[Cor 4.2]{l-w-x}
that $(X,T,\xi)$ is K-stable for any torus $T$ containing $\xi$
in its Lie algebra. Conversely, if $(X,T,\xi)$ is $K-$stable, then
it follows from \cite[Thm 6.1]{c.s0} that $\xi$ minimizes the volume
functional on the space of normalized Reeb vectors in the Lie algebra
of the maximal torus $T_{m}.$ 
\begin{rem}
More precisely, in \cite[Thm 6.1]{c.s0} it is assumed that $Y$ is
an isolated Gorenstein singularity, but, as pointed out in \cite[Remark 2.27]{l-x},
the result holds for a general $\Q-$Gorenstein affine variety $Y$
with klt singularities (and thus, in particular, for any toric affine
$\Q-$Gorenstein variety, since any such variety has klt singularities
\cite{dai}). 
\end{rem}

The proof of Theorem \ref{thm:toric CY iff vol min} is divided into
two parts. First, in the next section, we deduce from Theorem \ref{thm:main intro}
that $\xi$ minimizes the volume functional iff there exists a conical
Calabi-Yau metric on $T_{\C}\Subset Y$ with locally bounded potential
$r^{2}.$ Finally, the regularity of $r^{2}$ on all of $Y_{reg}$
is established in Section \ref{sec:Regularity-of-the}.

\subsection{\label{subsec:Existence-of-a conical bounded}Existence of a conical
Calabi-Yau metric on $T_{\C}\Subset Y$ with locally bounded potential
$r^{2}$}

Let $\mbox{Log }$be the standard map from $T_{\C}$ onto $\R^{m}$
defined by 
\begin{equation}
\mbox{Log}:\,T_{\C}\rightarrow\R^{m},\,\,z\mapsto x:=(\log(|z_{1}|^{2},...,\log(|z_{n}|^{2})\label{eq:def of Log map}
\end{equation}
so that the compact torus $T$ acts transitively on its fibers. We
will refer to $x$ as the \emph{logarithmic real coordinates} on $T_{\C}.$ 

By Prop \ref{prop:Q Gore} we can write
\begin{equation}
i^{m^{2}}\Omega\wedge\bar{\Omega}=\text{Log}^{*}\left(e^{\left\langle l,x\right\rangle }dx\right)\wedge d\theta,\label{eq:formula for Omega in terms of l}
\end{equation}
 where $d\theta$ denotes the invariant measure on $T.$ Accordingly,
if $r$ is conical with respect to $\xi,$ then, writing $r^{2}=\text{Log}^{*}f,$
the Calabi-Yau equation on $T_{\C}\Subset Y$
\[
\left(dd^{c}(r^{2})\right)^{m}=i^{m^{2}}\Omega\wedge\bar{\Omega},\,\,\,dd^{c}(r^{2})>0
\]
is equivalent to the real MA-equation for a smooth strictly convex
function $f$ in Theorem \ref{thm:MA for convex cone}. Next, note
that in the toric setting the volume of $\xi$ defined by formula
\ref{eq:V xi general intro}, coincides with the functional $V(\xi)$
defined by formula \ref{eq:def of V convex setting} in the convex-geometric
setting (as follows from formula \ref{eq:V as Laplac transf}). Moreover,
the bound \ref{eq:bound on f} on $f$ translates into to the bound
\[
r^{2}\leq C\exp\text{Log}^{*}\left(\sup_{p\in P_{\xi}}\left\langle x,p\right\rangle \right)
\]
As we will be explained next this bound implies that $r^{2}$ is locally
bounded in a neighborhood of the vertex point $y_{0}$ of $Y.$ Indeed,
it will be enough to show that the rhs in the previous inequality
extends from $\C^{*m}$ to a continuous function on $Y.$ Now, since
$P_{\xi}$ is a compact and convex polyhedron the sup in the rhs in
the formula in question can be written as the maximum of the corresponding
vertices of $\partial P_{\xi}.$ Since the maximum of a finite number
of continuous functions is continuous it will thus be enough to show
that that for a fixed $p\in P_{\xi}$ the function $F_{p}$ on $Y$
defined by $\exp\text{Log}^{*}\left(\left\langle x,p\right\rangle \right)$
is continuous on $Y.$ But since $p\in\mathcal{C}^{*}$ and $\mathcal{C}^{*}$
is the convex hull of $\mathcal{C}^{*}\cap M$ there exists a finite
number of points $p_{i_{1}},...,p_{i_{L}}$ in $\mathcal{C}^{*}\cap M$
that such that
\[
p=a_{1}p_{i_{1}}+\cdots a_{L}p_{i_{L}}
\]
Hence, $F_{p}$ is the pull-back to $Y$ under the embedding in Prop
\ref{prop:toric embedding} (for $C$ sufficiently large) of the following
continuous function on $\C^{N}$ with coordinates $Z_{i_{1}},...,Z_{i_{L}}:$
\[
|Z_{i_{1}}|^{2a_{1}}\cdots|Z_{i_{L}}|^{2a_{L}}
\]
and thus continuous, as desired. 

\subsection{\label{subsec:Variational-construction toric}Intermezzo: A variational
construction of $r^{2}$ }

We next make a short digression to point out that $r^{2},$ whose
existence was established in the precious section, can be constructed
by minimizing the Ding functional $\mathcal{D}$ (but this fact will
not be used in the sequel). 

Denote by $\mathcal{H}(Y)$ the space of all functions $\Phi$ on
$Y$ such that $e^{\Phi}$ is a $\xi-$conical Kähler potential on
$Y.$ First recall that exists a functional $\mathcal{E}$ on $\mathcal{H}(Y)$
determined, up to an additive constant, by the property that along
any affine curve $\Phi_{t}$ in $\mathcal{H}(Y)$ 
\[
\frac{d\mathcal{E}(\Phi_{t})}{dt}=\frac{1}{mV(\xi)}\int_{Y^{*}/\R^{*}}\frac{d\Phi_{t}}{dt}(dd^{c}\Phi_{t})^{m-1}\wedge d^{c}\Phi_{t}
\]
where the basic form $\frac{d\Phi_{t}}{dt}(dd^{c}\Phi_{t})^{m-1}\wedge d^{c}\Phi_{t}$
on $Y$ has been identified with a top form on the compact space $Y^{*}/\R^{*}$
(compare Section \ref{subsec:Transversal-a-priori}) and $V(\xi)$
denotes the volume of $\xi$ (this functional is denoted by $I$ in
the appendix of \cite{d-sII}). More generally, we denote by $\mathcal{E}$
the smallest usc extension of the functional $\mathcal{E}$ to the
space $PSH(Y,\xi)$ of all $\xi-$equivariant psh functions $\Phi$
on $Y.$ Set 
\[
\mathcal{E}^{1}(Y):=\left\{ \Phi\in PSH(Y,\xi):\,\,\mathcal{E}(\Phi)>-\infty\right\} 
\]
Now consider the\emph{ Ding functional }on $\mathcal{E}^{1}(Y)$ defined
by
\[
\mathcal{D}(\Phi):=-\frac{1}{m}\log\int_{Y^{*}/\R^{*}}i_{-J\xi}e^{-m\Phi}i^{m^{2}}\Omega\wedge\bar{\Omega}-\mathcal{E}(\Phi)
\]
 where the basic $(m-1)-$form $i_{-J\xi}e^{-(m+1)\Phi}i^{m^{2}}\Omega\wedge\bar{\Omega}$
on $Y$ has been identified with a measure on $Y^{*}/\R^{*}$ (the
functional $\mathcal{D}$ essentially coincides with the Ding functional
appearing in \cite{d-sII,C-S,l-w-x}). Next, we specialize to the
toric setting and denote by $\mathcal{E}^{1}(Y)^{T}$ the subspace
of all $T-$invariant functions in $\mathcal{E}^{1}(Y),$ where $T$
denotes the maximal torus. Then it is not hard to see that, under
the Log map \ref{eq:def of Log map}, the space $\mathcal{E}^{1}(Y)^{T}$
corresponds to the space $\mathcal{E}^{1}(\R^{m})$ introduced in
Section \ref{subsec:Variational-construction-of}. Moreover, the Ding
functional $\mathcal{D}(\Phi)$ corresponds to the functional $\mathcal{D}(\Phi)$
in formula \ref{eq:Ding in convex setting}. Thus Theorem \ref{thm:variational convex}
yields a variational construction of the potential $r^{2}$ of the
conical Calabi-Yau metric on $Y.$ 

\section{\label{sec:Regularity-of-the}Regularity of the Calabi-Yau metric
on the regular locus of $Y$}

Let us first come back to to the general setting of a good $T-$action
on a $\Q-$Gorenstein affine variety $(Y,T)$ (recalled in Section
\ref{subsec:Background}). Let $r$ be locally bounded psh solution
to the Calabi-Yau equation \ref{eq:CY intro} on $Y,$ in the usual
sense of local pluripotential theory. As pointed out in \cite{l-w-x}
one would expect that this implies that $r$ is, in fact, smooth on
the regular locus $Y_{reg}.$ Here we will provide a proof in the
toric setting. In fact, the proof can be generalized to the non-toric
setting, but since this would require a detour into pluripotential
theory (see Remark \ref{rem:non-toric gener}) the development of
the necessary pluripotential theory is left for the future.

\subsection{The general setup}

Let $r^{2}$ be a locally bounded psh $\xi-$conical function on $Y$
which satisfies the following equation on $Y:$
\begin{equation}
\left(dd^{c}(r^{2})\right)^{m}=i^{m^{2}}\Omega\wedge\bar{\Omega},\label{eq:MA eq for r in proof}
\end{equation}
in the usual sense of pluripotential theory.
\begin{rem}
If there exists a locally bounded psh solution $r^{2}$ as above then
it follows that $Y$ has log terminal (klt) singularities. Indeed,
by local pluripotential theory the integral of $\left(dd^{c}(r^{2})\right)^{m}$
over any neighborhood $V_{0}$ of $y_{0}$ with compact closure is
finite. Hence, $i^{m^{2}}\Omega\wedge\bar{\Omega}$ gives finite volume
to $V_{0},$ which, by general principles, implies that $Y$ has klt
singularities \cite{ko}.
\end{rem}

In fact, in the toric setting, where $T$ is has rank $m,$ we already
know from Section \ref{subsec:Existence-of-a conical bounded} that
$r$ is smooth on $T_{\C}\Subset Y_{reg}$ and it will be enough to
consider the equation there. Take a $T-$equivariant smooth resolution
$Y^{'}$ of $Y$ (as constructed in the proof of Lemma \ref{lem:barrier}
below) and denote by $p$ the corresponding projection: 
\[
p:\,Y'\rightarrow Y
\]
and by $D'$ the corresponding discrepancy divisor, i.e. the effective
$p-$exceptional $\Q-$divisor defined by the relation
\begin{equation}
p^{*}K_{Y}=K_{Y'}+D'\label{eq:discrepenct}
\end{equation}
(since $Y$ has Kawamata Log Terminal (klt) singularities the divisor
$D'$ is subklt, i.e the coeffiecents of $D'$ are $<1$). We may
assume that the divisor $D'$ is $T-$invariant and has simple normal
crossings. Denote by $\mathcal{U}$ the Zariski open subset of $Y'$
defined by
\[
\mathcal{U}:=p^{*}Y_{reg}
\]
Since the Reeb vector field $\xi$ may be identified with an element
in the Lie algebra of $T$ and $T$ acts on $Y'$ we can identify
$\xi$ with a vector field on $Y'.$ We will use a prime to indicate
pullbacks of functions and forms to $Y'.$ Accordingly, the equation
\ref{eq:MA eq for r in proof} induces the following equation on $\mathcal{U}:$
\begin{equation}
\left(dd^{c}(r')^{2}\right)^{m}=i^{m^{2}}\Omega'\wedge\overline{\Omega'},\,\label{eq:MA eq for r prime}
\end{equation}
Let $\varphi$ be the function on $Y$ defined by
\[
r^{2}=f_{\xi}e^{\varphi},
\]
where $f_{\xi}$ is a fixed conical Kähler potential on $Y$ (as furnished
by Lemma \ref{lem:family of smooth conical}) and denote by $\theta$
the following semi-positive smooth form on $Y^{*}:$ 
\[
\theta:=dd^{c}\log f_{\xi}.
\]
 Note that the function $\varphi$ is invariant under both $\xi$
and $J\xi.$

\subsection{A transversal Kähler metric $\omega_{B}$ on $Y'-\pi^{-1}(y_{0})$
and the corresponding $\xi-$equivariant barrier $\Phi_{B}$ on $\mathcal{U}$}

Let us first recall some general terminology in the context of foliations.
Let $S$ be a manifold and $\mathcal{F}$ a foliation on $S$ (i.e.
an integrable subbundle of $TS).$ The sheaf $\Omega_{B}^{\cdot}$
of\emph{ basic forms} on $S$ is defined as follows. Given an open
subset $U\Subset S$ the space $\Omega_{B}^{\cdot}(U)$ is defined
as the space of all differential forms $\alpha$ on $U$ such that
\[
\mathcal{L}_{V}\alpha=0,\,\,\,i_{V}\alpha=0
\]
 for any local vector field $V$ tangent to $\mathcal{F}$ (the definition
is made so that if the the set theoretic quotient $B:=S/\mathcal{F}$
is a manifold, then $\Omega_{B}^{\cdot}$ is the pull-back to $S$
of the sheaf of smooth forms on $B).$ In particular, $\Omega_{B}^{1}$
is naturally isomorphic to the sheaf of sections of the dual of the
normal bundle $TS/\mathcal{F}$ of $\mathcal{F}.$ The exterior derivative
$d$ on $S$ preserves $\Omega_{B}^{\cdot}.$ 

Now consider the case when $S:=Y'-\pi^{-1}(y_{0})$ and $\mathcal{F}$
is the foliation generated by the commuting vector fields $\xi$ and
$J\xi.$ Since $T\mathcal{F}$ is closed under the complex structure
$J$ on $S$ there is an induced complex structure on $TS/\mathcal{F}.$
A closed real $(1,1)-$form $\omega_{B}$ on $S$ will be called a
\emph{transversal Kähler form }if 
\[
\omega_{B}>0\,\text{\,on\,\,}TS/\mathcal{F},
\]
 i.e. if the symmetric form $\omega(\cdot,J\cdot)$ is positive definite
on $TS/\mathcal{F}$ (note that, since by assumption, $\omega_{B}$
is basic it descends to the quotient $TS/\mathcal{F}).$ Given a point
$p\in S$ there exists, by the inverse function theorem, local holomorphic
coordinates $(z,w)\in\C^{n}\times\C$ on $S$ centered at $p$ and
such that the derivations with respect to the imaginary and real parts
of $w$ are given by $\xi$ and $-J\xi$ respectively. In particular,
\[
\xi^{1,0}w=1,\,\,\,\xi^{1,0}z_{i}=0,\,\,\,\xi^{1,0}:=-\frac{1}{2}(J\xi+i\xi)
\]
 We will say that $(z,w)$ are \emph{holomorphic coordinates adapted
to $\xi$ }and call $z$ for the corresponding \emph{transversally
holomorphic coordinates. }

In order to construct a transversal Kähler form $\omega_{B}$ on $Y'-\pi^{-1}(y_{0})$
and an associated $\xi-$equivariant potential $\Phi_{B}$ (in the
sense of formula \ref{eq:xi equi}) on $\mathcal{U}\Subset Y'-\pi^{-1}(y_{0})$
(to be used as a barrier in the transversal a priori estimates) we
start with the following
\begin{lem}
\label{lem:barrier}There exists a $T-$invariant Kähler form $\omega$
on $Y'$ and a smooth $T-$invariant function $\Phi_{\omega}$ on
$\mathcal{U}$ such that $\Phi_{\omega}\rightarrow-\infty$ at $\partial\mathcal{U}$
and 
\[
dd^{c}\Phi_{\omega}=\omega
\]
 on $\mathcal{U}.$ 
\end{lem}

\begin{proof}
Take an equivariant embedding of $(Y,T)$ into $\C^{N},$ so that
$T$ corresponds to a linear torus action on $\C^{N}.$ Denote by
$\bar{Y}$ the Zariski closure of $Y$ in $\P^{N}$ under the standard
embedding of $\C^{N}$ into $\P^{N}.$ There exists a $T-$equivariant
resolution $\pi:\,\bar{Y'}\rightarrow\bar{Y}$ such that the inverse
image of the singular locus of $\bar{Y}$ coincides with the support
of an effective divisor $E$ on $X$ ($\pi$ can be taken as a strong
functorial resolution, as explained in \cite[Paragraphs 8,9]{ko2}).
Denote by $E_{i}$ the irreducible components of $E.$ Since $\pi$
is relatively ample there exist positive numbers $a_{i}$ such that
the $\Q-$line bundle $A:=\pi^{*}\mathcal{O}(1)-\sum a_{i}E_{i}$
over $\bar{Y'}$ is ample \cite{ha}. In other words, we get a decomposition
as $\Q-$line bundles 
\begin{equation}
\pi^{*}\mathcal{O}(1)=A+\sum_{i}a_{i}E_{i}\label{eq:decompo of L prime}
\end{equation}
 with $A$ ample and $a_{i}>0.$ Fix a $T-$invariant metric $h_{A}$
on $A$ whose curvature form defines a Kähler metric $\omega$ on
$\bar{Y'}$ and denote by $h$ the corresponding singular metric on
$\pi^{*}\mathcal{O}(1),$ induced from the decomposition \ref{eq:decompo of L prime}.
Now set $Y':=\pi^{-1}(Y).$ Since the restriction of $\pi^{*}\mathcal{O}(1)$
to $Y'$ is trivial - trivialized by the section $s$ obtained by
pulling back the standard trivialization of\emph{ $\mathcal{O}(1)$}
over $\C^{N}$ - the function function $\Phi_{\omega}:=-\log h(s)$
has the required properties. 
\end{proof}
We can now construct $\omega_{B}$ and the barrier $\Phi_{B}$ by
mimicking the construction of a Kähler metric on the symplectic quotient
of a Symplectic Kähler manifold $(S,\omega):$ 
\begin{prop}
\label{prop:red}There exists a transversal basic Kähler metric $\omega_{B}$
on $Y'-\pi^{-1}(y_{0})$ and an $\xi-$equivariant function $\Phi_{B}$
on $\mathcal{U}$ such that $\Phi_{B}\rightarrow-\infty$ at $\partial\mathcal{U}$
and 
\[
dd^{c}\Phi_{B}=\omega_{B}
\]
 on $\mathcal{U}.$
\end{prop}

\begin{proof}
First observe that the Reeb vector field $\xi$ is \emph{Hamiltonian}
with respect to the symplectic form defined by $\omega$ in the previous
lemma. This equivalently means that there exists a function $H$ on
$Y'-\pi^{-1}(y_{0})$ such that 
\begin{equation}
\nabla H=-J\xi,\label{eq:gradient eq}
\end{equation}
 where $\nabla$ denotes the gradient with respect to the Kähler metric
$\omega.$ Indeed, the existence of the Hamiltonian function $H$
follows from the fact that $\omega$ is the curvature of a Hermitian
metric on a line bundle $A$ over $Y'$ and that $\xi$ lifts to $A,$
preserving the Hermitian metric (and hence $H$ can be explicitly
defined as the derivative of the logarithm of the metric on $A$ along
$-J\xi).$ Since $-J\xi$ generates the $\R^{*}-$action on $Y'-\pi^{-1}(y_{0})$
the equation \ref{eq:gradient eq} implies that $H$ is strictly increasing
along the $\R^{*}-$action. Consequently, for $\lambda$ a sufficiently
large generic positive number,
\[
M_{\lambda}:=\left\{ H=\lambda\right\} 
\]
 is a compact smooth submanifold of $Y'-\pi^{-1}(y_{0}),$ which is
diffeomorphic to $\left(Y'-\pi^{-1}(y_{0})\right)/\R^{*}.$ Denote
by $\pi_{\lambda}$ the submersion from $Y'-\pi^{-1}(y_{0})$ onto
$M_{\lambda}$ (whose fibers are the $\R^{*}-$orbits) and by $i_{\lambda}$
the embedding of $M_{\lambda}$ into $Y'-\pi^{-1}(y_{0}):$ 
\[
\pi_{\lambda}:\,\,Y'-\pi^{-1}(y_{0})\rightarrow M_{\lambda},\,\,\,i_{\lambda}M_{\lambda}\hookrightarrow Y'-\pi^{-1}(y_{0})
\]
Now, given a point $p$ in $M_{\lambda}$ fix holomorphic coordinates
$(z,w)\in\C^{n}\times\C$ adapted to $\xi$ and defined on a neighborhood
$U_{p}$ of $p$ in $\mathcal{U}.$ Consider the following $\xi-$invariant
function on $M_{\lambda}\cap\mathcal{U}:$
\begin{equation}
\Phi_{\lambda}:=i_{\lambda}^{*}(\Phi-\lambda\Im w)\label{eq:def of Phi lambda}
\end{equation}
and the following $\xi-$invariant function on $U_{p}:$
\[
\Psi_{B}:=\pi_{\lambda}^{*}\Phi+\lambda\Im w
\]
In fact, this yields a well-defined function in a neighborhood of
$M_{\lambda}\cap\mathcal{U}.$ Indeed, if $(z',w')$ are $\xi-$adapted
holomorphic coordinates on $U_{p'}$ for $p'\in M_{\lambda},$ then
$w'-w=f(z)$ for some holomorphic function $f(z)$ of $z,$ defined
on on $U_{p}\cap U_{p'}$ (since $\xi^{1,0}(w'-w)=0).$ As a consequence,
\[
\Im w-i_{\lambda}^{*}(\Im w)=\Im w'-i_{\lambda}^{*}(\Im w'),
\]
 showing that $\Psi_{B}$ is well-defined in a neighborhood of $M_{\lambda}\cap\mathcal{U}.$
Next, observe that the Lie derivative of $\Psi_{B}$ along the generator
of the $\R^{*}-$action is given by
\begin{equation}
\mathcal{L}_{-JB}\Psi_{B}=\lambda.\label{eq:Lie deriv lambda}
\end{equation}
 Indeed, by construction, $\Psi_{\lambda}$ is $\R^{*}-$invariant
and $\mathcal{L}_{-JB}\Im w=1.$ As a consequence, the function $\Psi_{B}$
admits a unique smooth extension to all of $Y'-\pi^{-1}(y_{0})$ preserving
the property \ref{eq:Lie deriv lambda}. This means that the function
\begin{equation}
\Phi_{B}:=2\lambda^{-1}\Psi_{B}\label{eq:def of Phi B}
\end{equation}
is $\xi-$invariant and satisfies $\mathcal{L}_{-JB}\Phi_{B}=2.$
Next, observe that the two-form 
\[
\omega_{B}:=dd^{c}\Phi_{B}
\]
defines a transversal Kähler metric on $\mathcal{U}.$ This follows
from the fact that, locally, $\omega_{B}$ can be expressed as the
pull-back of a local Kähler metric $\omega_{\lambda}$ on $\C^{n},$
under the local map defined by the transversal holomorphic coordinates
$z.$ Indeed, a direct calculation, using the chain rule, reveals
that the restrictions of $\omega$ and $\omega_{B}$ to the hypersurface
$M_{\lambda}$ coincide. In fact, this is the same local computation
used in the construction of Kähler metrics on symplectic quotients,
where the local function $\Phi_{\lambda}$ appears as the local Kähler
potential of the reduced Kähler metric $\omega_{\lambda}$ (see the
proof of \cite[Formula 9.3]{b-g} and note that our local coordinate
$w$ corresponds to the logarithm of the coordinate on $\C^{*}$ used
in \cite{b-g}).

Finally, we note that $\omega_{B},$ originally defined on $\mathcal{U},$
extends to a transversal Kähler metric on $Y'-\pi^{-1}(y_{0}).$ Indeed,
this is shown by, locally, replacing the Kähler potential $\Psi$
of $\omega,$ used in the previous argument, with $\lambda^{-1}$
times any local $\xi-$invariant Kähler potential $\Phi_{U_{p}}$
for $\omega$ on a given open subset $U_{p}$ for $p\in M'.$ The
proof is now concluded by observing that $\Phi_{B}\rightarrow-\infty$
at $\partial\mathcal{U},$ as follows from the fact that $\lambda^{-1}\Phi-\Phi_{B}$
is bounded on $M_{\lambda}\cap\mathcal{U}$ (since $\lambda^{-1}\Phi-\Phi_{B}=\Im w-i_{\lambda}^{*}(\Im w)$
on $\mathcal{U}\cap U_{p}$ for any given $p\in M_{\lambda}$ and
using that $M_{\lambda}$ is compact. 
\end{proof}

\subsection{\label{subsec:Transversal-a-priori}Transversal a priori estimates}

We are now in a position to apply a transversal version of the a Laplacian
estimates in \cite[Appendix B]{bbegz} to the compact manifold 
\[
M':=\left(Y'-\pi^{-1}(y_{0})\right)/\R^{*},
\]
 where $\R^{*}$ denotes, as before, the $\R^{*}-$action generated
by $-J\xi.$ Denote by $\pi_{M'}$ the corresponding submersion from
$Y'-\pi^{-1}(y_{0})$ onto $M'.$ The Reeb vector field descends to
$M'$ and thus induces a foliation $\mathcal{F}$ on $M'.$ Using
the terminology recalled in the previous section we denote by $\Omega_{B}^{\cdot}$
the corresponding sheaf of\emph{ }basic forms on $M'.$ Hence, a form
$\alpha$ on $M'$ is basic iff
\[
\mathcal{L}_{\xi}\alpha=0,\,\,\,i_{\xi}\alpha=0.
\]
Moreover, the pull-back $\pi_{M'}^{*}$ induces a one-to-one correspondence
between basic forms on $Y'-\pi^{-1}(y_{0})$ and $M'.$ Next note
that the complex structure $J$ on $Y'$ induces a complex structure
on the normal bundle of $\mathcal{F}$ in $M',$ i.e. on $TM'/\R\xi$
(since the pullback of $TM'/\R\xi$ to $Y'-\pi^{-1}(y_{0})$ is naturally
isomorphic to the normal bundle of the foliation spanned by $\{\xi,J\xi\}$
in $Y'-\pi^{-1}(y_{0})$ and the action of $\R^{*}$ on $Y'-\pi^{-1}(y_{0})$
preserves $J).$ By duality we thus obtain a complex structure on
$\Omega_{B}^{1}(M')$ that we shall denote by the same symbol $J.$
Accordingly, we decompose 
\[
\Omega_{B}^{1}\otimes\C=\Omega_{B}^{1,0}\oplus\Omega_{B}^{0,1},
\]
as a sum of the eigenspaces corresponding to the eigenvalues $i$
and $-i$ of $J,$ respectively. The corresponding decomposition of
the exterior derivative yields
\[
d=:\partial_{B}+\bar{\partial}_{B},\,\,\,\partial_{B}:\,\Omega_{B}^{p,q}:=\left(\Omega_{B}^{1,0}\right)^{\wedge p}\wedge\left(\Omega_{B}^{0,1}\right)^{\wedge q}\rightarrow\Omega_{B}^{p+1,q}.
\]
 The pullback $\pi_{M'}^{*}$ induces a one-to-one correspondence
between basic $(p,q)-$forms on $Y'-\pi^{-1}(y_{0})$ and the sheaf
$\Omega_{B}^{p,q}$ on $M'.$ Moreover, under this correspondence
the operators $\partial_{B}$ and $\bar{\partial}_{B}$ on $M'$ correspond
to the ordinary operators $\partial$ and $\bar{\partial}$ on $Y'.$

If $z$ is a local transversally holomorphic coordinate centered at
$p\in Y'-\pi^{-1}(y_{0}),$ then $z_{i}$ descends to a local $\xi-$invariant
function on $M'$ that we shall call a\emph{ transversally holomorphic
coordinate on $M'.$ }We note the following elementary
\begin{lem}
\label{lem:basic}Let $\alpha$ be a basic $(1,1)-$form on $Y'-\pi^{-1}(y_{0}).$
Then the corresponding basic form on $M'$ has the property that,
locally, 
\[
\alpha=\pi_{\C^{n}}^{*}\tilde{\alpha},
\]
 where $\pi_{\C^{n}}$ denotes the local map to $\C^{n}$ defined
by $z$ and $\tilde{\alpha}$ is a local closed $(1,1)-$form on $\C^{n}.$
Conversely, any form $\alpha$ on $M'$ which locally can be expressed
in terms of $\tilde{\alpha},$ as above, is basic.
\end{lem}

\begin{proof}
First note that the local functions $z_{1},...z_{n}$ descend to define
local functions on $M'.$ Moreover, the corresponding local differentials
$dz_{1},...,dz_{N}$ span the sheaf $\Omega_{B}^{1,0}$ (since, $\Omega_{B}^{1,0}$
has rank $n$ and, by construction, $dz_{i}\in\Omega_{B}^{1,0}).$
Hence, if $\alpha$ is basic we can locally express $\alpha=\sum f_{ij}dz_{i}\wedge d\bar{z}_{j}.$
Using $\mathcal{L}_{\xi}\alpha=0=\mathcal{L}_{\xi}dz_{j}$ it follows
that $\mathcal{L}_{\xi}f_{ij}=0.$ Since $(z,\Im w)$ yield local
holomorphic coordinates on $M'$ and $\xi$ is the derivation with
respect to the imaginary part $\Im w$ of $w$ this shows that $\alpha$
can be expressed in terms of a form $\tilde{\alpha}$ as in the statement
of the lemma. The converse statement follows directly from the definitions. 
\end{proof}
We will say that a basic function $\psi:\,M'\rightarrow[-\infty,\infty[$
is\emph{ quasi-psh} if the function $\psi$ can be expressed as the
restriction to $M'$ of a local quasi-psh function on $Y',$ when
$M'$ is identified with the hypersurface $\{f_{\xi}=1\}$ in $Y'.$ 
\begin{lem}
\label{lem:psi plus minus}There exist two basic quasi-psh functions
$\Psi_{\pm}$ on $M'$ which are smooth on $\mathcal{U}$ such that
\[
(f_{\xi})^{-m}i^{m^{2}}\Omega'\wedge\bar{\Omega'})(\xi,J\xi,\cdot)=e^{m(\Psi_{+}-\Psi_{-})}\omega_{B}^{n}
\]
on $M'.$ Moreover, there exists a constant $A$ such that 
\[
\frac{i}{2\pi}\partial_{B}\bar{\partial}_{B}\Psi_{\pm}\geq-A\omega_{B}
\]
 holds in the sense of currents on $M'.$ 
\end{lem}

\begin{proof}
Decompose the discrepancy divisor $D'$ in formula \ref{eq:discrepenct}
as $D'=-D_{+}+D_{-}$ where $D_{\pm}$ are effective $T-$invariant
$\Q-$divisors (since $Y$ has klt singularities we can assumate that
the coefficents of $D_{-}$ are $<1$). It follows from formula \ref{eq:discrepenct}
that there exists a $T-$invariant volume form $dV$ on $Y',$ (multivalued)
$T-$equivairant holomorphic sections $s_{\pm}$ of the $\Q-$line
bundles over $Y'$ defined by $D_{\pm}$ and $T-$invariant metrics
$\left\Vert \cdot\right\Vert $ on $D_{\pm}$ such that 
\[
i^{m^{2}}\Omega'\wedge\bar{\Omega'})=\left\Vert s_{+}\right\Vert ^{2}\left\Vert s_{-}\right\Vert ^{-2}dV
\]
Now identify $M'$ with the hypersurface $\{f_{\xi}=1\}$ in $Y'.$
Then the formula stated in the lemma holds for some functions $\Psi_{\pm}$
such that $\Psi_{\pm}-\log\left\Vert s_{\pm}\right\Vert _{|M'}^{2}\in C^{\infty}(M').$
All that remains is thus to verify that there exists a constant $C$
such that 
\begin{equation}
\frac{i}{2\pi}\partial_{B}\bar{\partial}_{B}\log\left\Vert s_{\pm}\right\Vert _{|M'}^{2}\geq-C\omega_{B}\label{eq:pf of lemma Psi plus minus}
\end{equation}
 on $M'.$ To this end fix a point $p$ in $M$ and take holomorphic
coordinates $(z,w)$ adapted to $\xi$ and centered at $p.$ We can
then locally express the $T-$invariant functions $\left\Vert s_{\pm}\right\Vert ^{2}$
as $\left\Vert s_{\pm}\right\Vert ^{2}=|F_{\pm}(z,w)|^{2}e^{-\phi_{\pm}(z,w)}$
for local holomorphic functions $F_{\pm}$ and smooth $T-$invariant
functions $\phi_{\pm}.$ Since $F_{\pm}(z,w)$ are $T-$invariant
there exist $\lambda_{\pm}\in\R$ such that $\partial F_{\pm}/\partial\theta=i\lambda_{\pm}F_{\pm},$
where $\theta$ denotes the imaginary part of $w.$ After perhaps
replacing $F_{\pm}$ with $F_{\pm}e^{-\lambda w}$ we may thus assume
that $F_{\pm}$ is $\xi-$invariant and hence its restriction to $M'$
defines a local basic function on $M,$ which is annihilated by $\bar{\partial}_{B}.$
Hence, locally, we have
\[
\frac{i}{2\pi}\partial_{B}\bar{\partial}_{B}\log\left\Vert s_{\pm}\right\Vert _{|M'}^{2}\geq-\frac{i}{2\pi}\partial_{B}\bar{\partial}_{B}\phi_{\pm}
\]
 and since $\omega_{B}$ is transversely Kähler on $M'$ and $M'$
is compact this proves the inequality \ref{eq:pf of lemma Psi plus minus}. 
\end{proof}
Next, we observe that the function $\varphi'$ satisfies the following
``transversal Monge-Ampère equation'' on $M'\cap\mathcal{U}:$
\begin{lem}
\label{lem:transv eq}A conical smooth function $r^{2}$ satisfies
the equation \ref{eq:MA eq for r prime} on $Y_{reg}$ iff $\varphi'$
satisfies the following equation on $M'\cap\mathcal{U}$
\[
\left(\theta'+\frac{i}{2\pi}\partial_{B}\bar{\partial}_{B}\varphi'\right)^{n}=e^{-(n+1)\varphi'}e^{(n+1)(\Psi_{+}-\Psi_{-})}\omega_{B}^{n}
\]
where $\Psi_{\pm}$ are the functions appearing in the previous lemma.
\end{lem}

\begin{proof}
Setting $\Phi:=\log r^{2}$ gives $dd^{2}(r^{2})=e^{\Phi}(dd^{c}\Phi+d\Phi\wedge d^{c}\Phi)$
and hence the equation \ref{eq:MA eq for r prime} is equivalent to
\begin{equation}
(dd^{c}\Phi)^{n}\wedge d\Phi\wedge d^{c}\Phi=e^{-(n+1)\Phi}i^{m^{2}}\Omega'\wedge\bar{\Omega'}\label{eq:formula in pf transv eq}
\end{equation}
Next fix local holomorphic coordinates $(z,w)$ on $\mathcal{U}$
adapted to $\xi.$ Note that the local function $\phi:=\Phi-2\Im w$
is basic, i.e. satisfies $\xi\phi=(J\xi)\phi=0,$ using that $r^{2}$
is conical with respect to $\xi.$ Hence, we can identify $\phi$
with a function of $z$ and write $\phi=\phi(z),$ abusing notation
slightly. Since $dd^{c}\Im w=0$ we can thus locally express
\[
(dd^{c}\Phi)^{n}\wedge d\Phi\wedge d^{c}\Phi=\left(d_{z}d_{z}^{c}\phi(z)\right)^{n}\wedge d(\Im w)\wedge d^{c}(\Im w)=\left(dd^{c}\Phi\right)^{n}\wedge d(\Im w)\wedge d^{c}(\Im w)
\]
As a consequence, contracting the equation \ref{eq:MA eq for r in proof}
with first $-J\xi$ and then $\xi$ yields 
\[
(dd^{c}\Phi)^{n}=(d_{z}d_{z}^{c}\Phi)^{n}=e^{-(n+1)\varphi}((f_{\xi})^{-2m}i^{m^{2}}\Omega\wedge\bar{\Omega})(\xi,J\xi,\cdot)
\]
 on $\mathcal{U}.$ Since $dd^{c}\Phi=\theta+dd^{c}\varphi$ on $Y$
the previous argument reveals that $\theta$ is basic on $\mathcal{U}$
and that the equation \ref{eq:MA eq for r in proof} implies the transversal
equation stated in the lemma. The converse statement follows in a
similar way, using that the measure appearing in the rhs of formula
\ref{eq:formula in pf transv eq} is invariant under both $\xi$ and
$J\xi.$ 
\end{proof}
\begin{rem}
More generally, the previous lemma holds if $\phi'$ is merely assumed
bounded (and $\theta'-$psh) if one defines the basic current $\left(\theta_{B}'+\frac{i}{2\pi}\partial_{B}\bar{\partial}_{B}\varphi'\right)^{n}$
on $M'$ to be the one corresponding by pull-back to the $(n,n)-$current
defined in the usual sense of local pluripotential theory on $Y'$
(i.e. in the local sense of Bedford-Taylor). 
\end{rem}

Setting

\[
\text{tr}_{\omega_{B}}\alpha:=n\frac{\alpha\wedge\omega_{B}^{n-1}}{\omega_{B}^{n}},
\]
for a given basic $(1,1)-$form $\alpha,$ we have the following transversal
generalization of the Aubin-Yau Laplacian inequality in Kähler geometry
(in Siu's form, as stated in \cite[Appendix B]{bbegz}):
\begin{lem}
\label{lem:siu}Let $\omega$ be a Kähler form on a neighborhood of
$M'$ and denote by $\kappa$ the sup over $M'$ of the absolute value
of the holomorphic bisectional curvatures of $\omega.$ If $\widetilde{\omega}{}_{B}$
is a transversal Kähler form on $M',$ then
\[
\text{tr}_{\omega_{B}}\left(\frac{i}{2\pi}\partial_{B}\bar{\partial}_{B}\log(\text{tr}_{\omega_{B}}\widetilde{\omega}{}_{B})\right)\geq-\frac{\text{tr}_{\omega_{B}}(\text{Ric \ensuremath{\widetilde{\omega}{}_{B})}}}{\text{tr}_{\omega_{B}}\ensuremath{\omega'}_{B}}-\kappa\text{tr}_{\widetilde{\omega}_{B}}\omega_{B}
\]
\end{lem}

\begin{proof}
Thanks to Lemma \ref{lem:basic} this follows from the usual Laplacian
estimate in Kähler geometry (as stated in \cite[Appendix B]{bbegz}),
since the proof of the latter inequality is purely local. In fact,
this shows that $\kappa$ can, more precisely, be taken to be a lower
bound on the ``transversal'' bisectional curvatures of $\omega_{B}.$ 
\end{proof}
With these preparations in place we are now ready to repeat the argument
in \cite[Appendix B]{bbegz}. Set
\[
\psi_{+}:=\Psi_{+},\,\,\,\psi_{-}:=\Psi_{-}+\varphi',
\]
The functions $\psi_{\pm}$ still satisfy an inequality as in Lemma
\ref{lem:psi plus minus}, since $\theta'$ is smooth and basic. The
equation appearing in Lemma \ref{lem:transv eq} can thus be expressed
as 
\[
\left(\theta'+\frac{i}{2\pi}\partial_{B}\bar{\partial}_{B}\varphi'\right)^{n}=e^{(n+1)(\psi_{+}-\psi_{-})}\omega_{B}^{n}.
\]
Now take two sequences of smooth functions $\psi_{\pm,j}$ on $M'$
decreasing to $\psi_{\pm},$ respectively and such that 
\begin{equation}
\frac{i}{2\pi}\partial_{B}\bar{\partial}_{B}\psi_{\pm,j}\geq-C\omega_{B},\label{eq:lower bound ddc psi j}
\end{equation}
 on $M'$ (the existence of such sequences follows, for example, from
the regularization procedure introduced in\cite{be0}, adapted to
the present setting, by solving transversal Monge-{}-Ampère equation
on $(M',\omega_{B})$ involving a large parameter $\beta).$ 

Next, fix $\epsilon>0$ and consider the following equation on $M'$
where $\omega$ is the Kähler form in Lemma \ref{lem:barrier}:
\begin{equation}
\left(\theta'+\epsilon\omega_{B}+\frac{i}{2\pi}\partial_{B}\bar{\partial}_{B}\varphi_{j,\epsilon}\right)^{n}=e^{(n+1)(\psi_{+,j}-\psi_{-,j})}\omega_{B}^{n},\label{eq:transv MA eq with eps}
\end{equation}
 for a smooth $\xi-$invariant function $\varphi_{j,\epsilon},$ satisfying
\[
\theta'+\epsilon\omega_{B}+\frac{i}{2\pi}\partial_{B}\bar{\partial}_{B}\varphi_{j,\epsilon}>0
\]
By the transversal generalization of the Calabi-Yau theorem in \cite{el}
such a function exists and is uniquely determined modulo an additive
constant that we fix by requiring that the sup of $\varphi_{j,\epsilon}$
over $M'$ vanishes. Using the $\R^{*}-$action generated by $-J\xi$
we can identify $\varphi_{j,\epsilon}$ with a smooth function on
$\mathcal{U}$ which is invariant under both $\xi$ and $J\xi.$ 
\begin{lem}
\label{lem:L infty estimate toric with eps}There exists a uniform
constant $C$ such that 
\[
\left\Vert \varphi_{j,\epsilon}\right\Vert _{L^{\infty}(M')}\leq C
\]
\end{lem}

\begin{proof}
In the toric setting this follows from Lemma \ref{lem:L infty estimate in convex settting}.
To see this first note that if $T$ is taken to be the compact maximal
torus acting on the toric variety $Y,$ then $Y'$ is a non-singular
toric variety, which yields a toric resolution $Y'\rightarrow Y.$
We can thus identify $T_{\C}$ with an open subset of $Y'.$ Proceeding
as in the proof of Theorem \ref{thm:MA for convex cone} we have a
fibration 
\begin{equation}
\text{Log:\,\,}T_{\C}/\left(\R\xi+\R J\xi\right)\rightarrow\R^{n+1}/\R\xi\simeq\R^{n}\label{eq:log quotient}
\end{equation}
Indeed, we can take local $\xi-$adapted holomorphic coordinates $(z,w)$
on $T_{\C}$ as in Lemma \ref{lem:basic}, which have the property
that 
\[
s:=(\log|z_{1}|^{2},...,\log|z_{n}|^{2})\in\R^{n}
\]
 coincides with $s,$ appearing in formula \ref{eq:x maps to}. Under
this fibration we can express $\theta_{B}'$ and $\omega_{B}$ on
$T_{\C}/\left(\R\xi+\R J\xi\right)$ in terms of smooth convex functions
$f$ and $g$ on $\R^{n}$ (with coordinates $s):$
\[
\theta'=d_{B}d_{B}^{c}\text{Log}^{*}f,\,\,\,\omega_{B}=d_{B}d_{B}^{c}\text{Log}^{*}g
\]
 As a consequence, 
\[
\theta'+\epsilon\omega_{B}=d_{B}d_{B}^{c}\text{Log}^{*}(f+\epsilon g)
\]
As will be shown below the closure of $\left(\nabla(f+\epsilon g)\right)(\R^{n})$
is a bounded convex body in $\R^{n}$ of the form 
\[
P_{\epsilon}:=P_{f}+\epsilon P_{g},
\]
 where $P_{f}$ and $P_{g}$ are bounded convex polytopes arising
as the closures of $(\nabla f)(\R^{n})$ and $\left(\nabla g\right)(\R^{n}),$
respectively. We can thus express 
\[
\varphi_{j,\epsilon}=\phi_{j,\epsilon}-(f+\epsilon g),
\]
 for a smooth convex function $\phi_{j,\epsilon}$ on $\R^{n}$ satisfying
\[
(\det\nabla_{s}^{2})(\phi_{j,\epsilon})ds=\mu,\,\,\,\overline{\nabla\phi_{j,\epsilon}(\R^{n})}=P_{\epsilon},
\]
where the measure $\mu$ on $\R^{n}$ is the volume form on $\R^{n}$
corresponding to the basic form $i^{m^{2}}r^{-2m}\Omega\wedge\bar{\Omega}$
under the fibration \ref{eq:log quotient} and thus has exponential
decay (by formula \ref{eq:exp decay}). Hence, by the $L^{\infty}-$estimate
in Lemma \ref{lem:L infty estimate in convex settting} 
\[
\left\Vert \phi_{j,\epsilon}-\phi_{P_{\epsilon}}-\sup_{\R^{n}}(\phi_{j,\epsilon}-\phi_{P_{\epsilon}})\right\Vert _{L^{\infty}(\R^{n})}\leq C'
\]
 for a uniform constant $C'$ (recall that the support function of
a convex set $P$ is denoted by $\phi_{P}).$ Note that 
\[
\phi_{P_{\epsilon}}=\phi_{P_{f}}+\epsilon\phi_{P_{g}}
\]
 since, in general, the supporting function of a Minkowski sum of
convex sets is the sum of the supporting functions. Since we have
assumed that $\sup_{\R^{n}}(\phi_{j,\epsilon}-(f+\epsilon g))=0,$
this means that in order to conclude that $\left\Vert \phi_{j,\epsilon}-(f+\epsilon g)\right\Vert _{L^{\infty}(\R^{n})}$
is uniformly bounded, all that remains is to verify that

\begin{equation}
\left\Vert \phi_{P_{f}}-f\right\Vert _{L^{\infty}(\R^{n})}\leq C_{f},\,\,\,\left\Vert \phi_{P_{g}}-g\right\Vert _{L^{\infty}(\R^{n})}\leq C_{g}.\label{eq:bounds Cf Cg}
\end{equation}
To this end first recall that, by construction, $\exp(f(s)+t)$ is
the convex function on $\R^{m}$ corresponding to a $\xi-$conical
Kähler potential on $Y.$ Moreover, as explained in Section \ref{subsec:Existence-of-a conical bounded},
the function $\exp\phi_{P_{\xi}}(x)$ on $\R^{n}$ (which may be expressed
as $\exp(\phi_{Q_{\xi}}(s)+t),$ where $Q_{\xi}$ is defined below
formula \ref{eq:def of Q}) corresponds to a $\xi-$conical continuous
function on $Y.$ Hence, $\exp(f(s)+t)/\exp(\phi_{Q_{\xi}}(s)+t)$
extends to a continuous function on the compact set $Y/\R^{*}$ and
is thus bounded. It follows that $f-\phi_{Q_{\xi}}$ is bounded on
$\R^{n}$ and hence the first bound in formula \ref{eq:bounds Cf Cg}
holds with $P_{f}=Q_{\xi}.$ 

To prove the second bound in formula \ref{eq:bounds Cf Cg} recall
that $\omega_{B}$ is defined as the basic form in Prop \ref{prop:red}
constructed from the $T-$invariant Kähler form $\omega$ on $Y'$
in Lemma \ref{lem:barrier}. In turn, in the toric setting $\omega$
is the restriction to $Y'$ of the curvature form of a toric smooth
metric $h_{A}$ on a toric line bundle $A\rightarrow\bar{Y'}$ over
toric projective manifold. To simplify the notation we will assume
that $\lambda=1$ in formula \ref{eq:def of Phi B}, but the general
case is shown in essentially the same way, up to a scaling. By standard
projective toric geometry $h_{A}$ corresponds to a convex function
$g_{A}(s,t)$ on $\R^{m}$ such that $\overline{\nabla g_{A}(\R^{n})}=P_{A},$
where $P_{A}$ is a bounded convex polytope in $\R^{m}:$ 
\[
\text{Log}^{*}g_{A}=\Phi:=-\log h_{A}(s_{0},s_{0}),
\]
 where $s_{0}$ is the standard invariant trivializing holomorphic
section of $A$ over $T_{\C}\Subset Y'.$ Moreover, the Legendre-Fenchel
transform $g_{A}^{*}$ of $g_{A}$ is bounded on $P_{A}$ and equal
to $\infty$ on the complement of $P_{A}$ \cite[Prop 3.3]{be-be}.
All that remains is thus to verify the following 

\[
\text{Claim:\,\,\,}(i)\,P_{g}=P_{A}\cap\{\left\langle \xi,\cdot\right\rangle =1\},\,\,\,(ii)\,g^{*}(p)=(g_{A}(s,t))^{*}(p,1)
\]
where $\xi$ is identified with the unit-vector along the $t-$axis
in $\R_{s}^{n}\times\R_{t}$ and $g^{*}(p)$ denotes the Legendre-Fenchel
transform of $g$ on $\R^{n}.$ Indeed, accepting this claim we get
\begin{equation}
\left\Vert \phi_{P_{g}}-g\right\Vert _{L^{\infty}(\R^{n})}=\left\Vert g^{*}\right\Vert _{L^{\infty}(P_{g})}\leq\left\Vert g_{A}^{*}\right\Vert _{L^{\infty}(P_{A})}:=C_{A}<\infty,\label{eq:pf L infty est}
\end{equation}
proving the second bound in \ref{eq:bounds Cf Cg} with $P_{g}$ equal
to the polytope in Claim $(i)$ and $C_{g}=C_{A}$ (the first identity
in formula \ref{eq:pf L infty est} follows from the basic properties
of the Legendre transforms; see \cite[Prop 2.3]{be-be}). Finally,
to prove the Claim above, first note that in terms of the coordinates
$(s,t)$ the Hamiltonian function appearing in formula \ref{eq:gradient eq}
is given by 
\[
H(s,t)=\partial_{t}g_{A}(s,t)
\]
 (as follows directly from the definition of $H$). Hence, by the
definition \ref{eq:def of Phi lambda}, 
\[
g(s)=g_{A}(s,t)-t,\,\,\,t=\partial_{t}g_{A}(s,t)
\]
 But this means that 
\[
g(s)=\inf_{t}\left(g_{A}(s,t)-t\right)
\]
Hence, 
\[
g^{*}(p)=\sup_{s,t}\left(\left\langle p,s\right\rangle -(g_{A}(s,t)-t)\right)=\sup_{s,t}\left(\left\langle p,s\right\rangle +1\cdot t-g_{A}(s,t)\right)=:g_{A}^{*}(p,1),
\]
 which proves $(ii)$ and then $(i)$ follows from the fact that the
closure of the gradient image of a convex function $g$ coincides
with the locus where $g^{*}<\infty$ (by standard convex analysis).
\end{proof}
Now set

\[
\omega_{B}^{\epsilon}:=\theta'+\epsilon\omega_{B},\,\,\,\widetilde{\omega}{}_{B}:=\theta'+\epsilon\omega_{B}+dd^{c}\varphi_{j,\epsilon}
\]
and 
\[
\psi:=\Phi_{B}-f_{\xi},
\]
 which is a basic function on $\mathcal{U}$ which, by Prop \ref{prop:red},
has the property that $\psi\rightarrow-\infty$ at the boundary of
$M'\cap\mathcal{U}$ and $\frac{i}{2\pi}\partial_{B}\bar{\partial}_{B}\psi+\theta'\geq0.$
Consider the following smooth function on $M'\cap\mathcal{U}:$
\[
h:=\log(\text{tr}_{\omega_{B}^{\epsilon}}\widetilde{\omega}{}_{B})+(n+1)\psi_{-,j}-A_{1}(\varphi_{j,\epsilon}-\psi)
\]
 for $A_{1}$ a constant (to be chosen sufficiently large). Since
$\varphi_{j,\epsilon}$ is bounded and $\psi\rightarrow-\infty$ at
the boundary of $M'\cap\mathcal{U}$ the maximum of $h$ is attained
in $M'\cap\mathcal{U}.$ Combining the inequality in Lemma \ref{lem:siu}
with the transversal MA-equation \ref{eq:transv MA eq with eps} and
the uniform $L^{\infty}-$estimate in Lemma \ref{lem:L infty estimate toric with eps}
then yields, precisely as in \cite[Appendix B]{bbegz}, 
\begin{equation}
\sup_{M'\cap\mathcal{U}}\text{tr}_{\omega_{B}^{\epsilon}}\widetilde{\omega}{}_{B}\leq A_{2}e^{-\psi_{-,j}}e^{-A_{1}\psi}\leq A_{2}e^{-\psi_{-}}e^{-A_{1}\psi}\label{eq:sup tr}
\end{equation}
 for some constants $A_{1}$ and $A_{2}$ (only depending on the lower
bound in formula \ref{eq:lower bound ddc psi j}, the upper bound
on the $L^{\infty}-$norms in Lemma\ref{lem:L infty estimate toric with eps}
and the bound $\kappa$ on the bisectional curvatures of the reference
Kähler form $\omega$). 

\subsection{Conclusion of the proof of the regularity of $r^{2}$ on $Y_{reg}$}

By construction the rhs in the estimate \ref{eq:sup tr} is locally
bounded on $M'\cap\mathcal{U}.$ Hence, locally on $\mathcal{U},$
the functions $\varphi_{j,\epsilon}$ have uniformly bounded Laplacians$.$
As a consequence, there exists a subsequence $\varphi_{j,\epsilon(j)}$
converging in the local $C^{1}-$topology on $\mathcal{U}$ to a function
$\phi_{0}$ $\in L^{\infty}(\mathcal{U})$ such that $dd^{c}\phi_{0}$
is locally bounded on $\mathcal{U},$ invariant under both $\xi$
and $J\xi$ and satisfies the ``inhomogeneous'' MA-equation
\begin{equation}
(\theta+dd^{c}\phi_{0})^{m}=i^{m^{2}}r^{-2m}\Omega\wedge\overline{\Omega},\,\,\,dd^{c}\Phi_{0}\geq0\label{eq:inhomogen ma eq on Y}
\end{equation}
on $\mathcal{U},$ identified with $Y_{reg}$ (using the same local
computation as in the proof of  Lemma \ref{lem:transv eq}). But there
is a unique such solution, modulo the addition of a constant. Indeed,
in the toric case this follows from the well-known uniqueness, modulo
additive constants, of solution to the real Monge-{}-Ampère equation
in Lemma \ref{lem:L infty estimate in convex settting} (see \cite[Lemma 3.1]{be}
for a simple proof). Hence, $\varphi_{0}=\varphi'+C.$ This shows
that $\varphi'$ has a locally bounded Laplacian on $\mathcal{U}.$
Thus, applying standard local Evans-Krylov-Trudinger theory to the
fully non-linear PDE \ref{eq:MA eq for r prime} on\emph{ $\mathcal{U}$}
shows that $\varphi'$ is, in fact, smooth on $\mathcal{U}$ (by \cite[Thm 2.6]{bl}
it is enough to know that $\Delta\varphi'\in L_{loc}^{p}$ for $p>2m(m-1)$).
This means that $r^{2}$ is smooth on $Y_{reg},$ as desired. 
\begin{rem}
\label{rem:non-toric gener}The only place where the toric structure
was used was in the $L^{\infty}-$bound in Lemma \ref{lem:L infty estimate toric with eps}
and the uniqueness of bounded pluripotential basic solutions to the
equation $(\theta+dd^{c}\phi_{0})^{m-1}=\nu$, for $\nu$ a given
basic $(m-1,m-1)-$form on $Y_{reg}$ such that the quotient of $\nu$
and $f_{\xi}^{-m}i^{m^{2}}\Omega\wedge\bar{\Omega}(\xi,J\xi,\cdot)$
is bounded. 
\end{rem}

\end{document}